\documentclass[a4paper]{article}

\newcommand{\NOTESON}{0}
\newcommand{\Debug}{0}
\newcommand{\COLORON}{0}

\usepackage[usenames,dvips]{color} 
\usepackage{amsthm,amssymb,amsmath,enumerate,graphicx,epsf,mathbbol,stmaryrd}
\usepackage[dvips, bookmarks, colorlinks=false]{hyperref}

\newcommand{\comment}[1]{}
\newcommand{\COMMENT}[1]{}

\newcommand{\defi}[1]{{\color{VioletRed}\emph{#1}}}

\newcommand{\acknowledgement}{\section*{Acknowledgement}}

\comment{
	\begin{lemma}\label{}	
	\end{lemma}
	\begin{proof}

	\end{proof}

}

\newtheorem{proposition}{Proposition}[section]

\newtheorem{theorem}[proposition]{Theorem}
\newtheorem{corollary}[proposition]{Corollary}
\newtheorem{lemma}[proposition]{Lemma}
\newtheorem{observation}[proposition]{Observation}
\newtheorem{conjecture}{{\color{red}Conjecture}}[section]

\newtheorem{problem}[conjecture]{{\color{red}Problem}}
\newtheorem*{noTheorem}{Theorem}
\newtheorem{examp}[proposition]{Example}

\newcommand{\example}[2]{\begin{examp} \label{#1} {{#2}}\end{examp}}

\newcommand{\kreis}[1]{\mathaccent"7017\relax #1}

\newcommand{\FIG}{0}

\ifnum \NOTESON = 1 \newcommand{\note}[1]{ 

	\ 

	{\color{blue} \hspace*{-60pt} NOTE: \color{Turquoise}{\small  \tt \begin{minipage}[c]{1.1\textwidth}  #1 \end{minipage} \ignorespacesafterend }} 
	
	\ 
	
	}
\else \newcommand{\note}[1]{} \fi

\ifnum \Debug = 1 \newcommand{\sss}{\ensuremath{\bowtie \bowtie \bowtie\ }}
\else \newcommand{\sss}{} \fi

\ifnum \FIG = 1 \newcommand{\fig}[1]{Figure ``{#1}''}
\else \newcommand{\fig}[1]{Figure~\ref{#1}} \fi

\ifnum \Debug = 1 \usepackage[notref,notcite]{showkeys}
\fi

\ifnum \COLORON = 0 \renewcommand{\color}[1]{}
\fi

\newcommand{\showFig}[2]{
   \begin{figure}[htbp]
   \centering
   \noindent
   \epsfbox{#1.eps}
   \caption{\small #2}
   \label{#1}
   \end{figure}
}

\newcommand{\N}{\mathbb N}
\newcommand{\R}{\mathbb R}
\newcommand{\Z}{\mathbb Z}
\newcommand{\Q}{\mathbb Q}

\newcommand{\cf}{\mathcal F}

\newcommand{\ch}{\mathcal H}
\newcommand{\ci}{\mathcal I}

\newcommand{\cp}{\mathcal P}

\newcommand{\cu}{\mathcal U}

\newcommand{\oo}{\ensuremath{\omega}}
\newcommand{\OO}{\ensuremath{\Omega}}
\newcommand{\alp}{\ensuremath{\alpha}}

\newcommand{\del}{\ensuremath{\delta}}
\newcommand{\eps}{\ensuremath{\epsilon}}

\newcommand{\sig}{\ensuremath{\sigma}}

\newcommand{\el}{\ensuremath{\ell}}

\newcommand{\ccc}{\ensuremath{\mathcal C}}
\newcommand{\ccg}{\ensuremath{\mathcal C(G)}}
\newcommand{\fcg}{\ensuremath{|G|}}

\newcommand{\zero}{\mathbb 0}
\newcommand{\sm}{\backslash}
\newcommand{\restr}{\upharpoonright}

\newcommand{\nin}{\ensuremath{{n\in\N}}}
\newcommand{\iin}{\ensuremath{{i\in\N}}}
\newcommand{\unin}{\ensuremath{[0,1]}}

\newcommand{\pth}[2]{\ensuremath{#1}\text{--}\ensuremath{#2}~path}

\newcommand{\pths}[2]{\ensuremath{#1}\text{--}\ensuremath{#2}~paths}
\newcommand{\arc}[2]{\ensuremath{#1}\text{--}\ensuremath{#2}~arc}
\newcommand{\arcs}[2]{\ensuremath{#1}\text{--}\ensuremath{#2}~arcs}
\newcommand{\seq}[1]{\ensuremath{(#1_i)_{i\in\N}}} 
\newcommand{\sseq}[2]{\ensuremath{(#1_i)_{i\in #2}}} 
\newcommand{\fml}[1]{\ensuremath{\{#1_i\}_{i\in I}}} 
\newcommand{\flo}[2]{\ensuremath{#1}\text{--}\ensuremath{#2}~flow} 

\newcommand{\g}{\ensuremath{G\ }}
\newcommand{\G}{\ensuremath{G}}

\newcommand{\s}{s}

\newcommand{\half}{\frac{1}{2}}

\newcommand{\Lr}[1]{Lemma~\ref{#1}}
\newcommand{\Tr}[1]{Theorem~\ref{#1}}
\newcommand{\Sr}[1]{Section~\ref{#1}}

\newcommand{\Cr}[1]{Corollary~\ref{#1}}
\newcommand{\Cnr}[1]{Con\-jecture~\ref{#1}}

\newcommand{\Er}[1]{Example~\ref{#1}}

\newcommand{\lf}{locally finite}
\newcommand{\lfg}{locally finite graph}

\newcommand{\bos}{basic open set}

\newcommand{\nlf}{non-locally-finite}
\newcommand{\nlfg}{non-locally-finite graph}
\newcommand{\fhcy}{Hamilton cycle}
\newcommand{\hcy}{Hamilton circle}
\newcommand{\et}{Euler tour}

\renewcommand{\iff}{if and only if}
\newcommand{\fe}{for every}

\newcommand{\st}{such that}

\newcommand{\ti}{there is}

\newcommand{\obda}{without loss of generality}
\newcommand{\btdo}{by the definition of}

\newcommand{\wrt}{with respect to}

\newcommand{\inm}{infinitely many}

\newcommand{\ises}{is easy to see}

\newcommand{\FC}{Freudenthal compactification}
\newcommand{\tcs}{topological cycle space}
\newcommand{\tocir}{topological circle}
\newcommand{\tet}{topological Euler tour}

\newcommand{\labtequ}[2]{ \begin{equation} \label{#1} 	\begin{minipage}[c]{0.9\textwidth}  #2 \end{minipage} \ignorespacesafterend \end{equation} }

\newcommand{\mymargin}[1]{
  \marginpar{%
    \begin{minipage}{\marginparwidth}\small%
      \begin{flushleft}%
        {\color{blue}#1}%
      \end{flushleft}%
   \end{minipage}%
  }%
}%

\newcommand{\afsubm}[1]{ \ifnum \Debug = 1 {\mymargin{#1}}
\fi}

\newcommand{\mySection}[2]{}

\newcommand{\DK}{Diestel and K\"uhn}
\newcommand{\BS}{Bruhn and Stein}

\newcommand{\RD}{Diestel}

\newcommand{\CDB}{\cite[Section~8.5]{diestelBook05}}

\newcommand{\LemArcC}{\cite[p.~208]{ElemTop}}
\newcommand{\LemArc}[1]{
	\begin{lemma}[\LemArcC]
	\label{#1}
	The image of a topological path with endpoints $x,y$ in a Hausdorff space $X$ contains an arc in $X$ between $x$ and $y$.
	\end{lemma}
}

\newcommand{\LemCombStarC}{\cite[Lemma~8.2.2]{diestelBook05}} 
\newcommand{\LemCombStar}[1]{
	\begin{lemma}[\LemCombStarC]
	\label{#1}
	Let $U$ be an infinite set of vertices in a connected graph $G$. Then $G$ contains either a ray $R$ and infinitely many pairwise disjoint \pths{U}{R} or a subdivided star with infinitely many leaves in $U$. 
	\end{lemma}
}

\newcommand{\LemCycDecC}{\cite[Theorem~8.5.8]{diestelBook05}} 
\newcommand{\LemCycDec}[1]{
	\begin{lemma}[\LemCycDecC]
	\label{#1}
	Every element of $\ccg$ is a disjoint union of circuits.
	\end{lemma}
}

\newcommand{\LemInjHomC}{\cite{armstrong}} 
\newcommand{\LemInjHom}[1]{
	\begin{lemma}[\LemInjHomC] \label{#1}
	A continuous bijection from a compact space to a Hausdorff space is a homeomorphism.
	\end{lemma}
}

\newcommand{\ThmMacLaneC}{\cite{maclane37}}
\newcommand{\ThmMacLane}[1]{
	\begin{theorem}[\ThmMacLaneC]
	\label{#1}
A finite graph is planar if and only if its cycle space \ccg\ has a simple generating set.
	\end{theorem}
}

\newcommand{\gid}[1]{\overline{#1}}

\newcommand{\ltopf}[1]{\ensuremath{#1\text{-}TOP}}
\newcommand{\ltop}{\ensuremath{\ltopf{\ell}}}
\newcommand{\ltopx}[1]{\ensuremath{ |#1|_\ell }}
\newcommand{\ltopxl}[2]{\ensuremath{|#2|_{#1} }}
\newcommand{\ltp}{\ltopx{$G$}}

\newcommand{\blg}{\ensuremath{\partial^\ell G}}
\newcommand{\bxx}[2]{\ensuremath{\left<#2\right>_#1 }}
\newcommand{\blom}{\bxx{\ell}{\oo}}

\newcommand{\mtop}{\ensuremath{MTOP}}
\newcommand{\mtp}{\mtop($G$)}
\newcommand{\mtopx}[1]{MTOP\ensuremath{ ({#1}) } }
\newcommand{\lER}{\ensuremath{\ell: E(G) \to \R_{>0}}}
\newcommand{\etop}{\ensuremath{ETOP}}
\newcommand{\eti}{\ensuremath{\etop(G)}}
\newcommand{\GetopG}{\ensuremath{ETOP'}}
\newcommand{\Getop}{\ensuremath{\GetopG(G)}}
\newcommand{\ccl}{\ensuremath{\mathcal C_\ell}}
\newcommand{\cclg}{\ensuremath{\mathcal H_\ell(G)}}
\newcommand{\geol}{\ensuremath{\mathcal G_\ell(G)}}
\newcommand{\EOO}{\ensuremath{\OO'}}

\newcommand{\eqT}{\approx}

\newcommand{\ltpn}{\ensuremath{\ltpr{n}}}
\newcommand{\ltpr}[1]{\ensuremath{\ltp \sm \kreis{E_{#1}}}}

\newcommand{\lL}{\ell_L}

\newcommand{\lhom}{\ensuremath{H_\ell}}
\newcommand{\lhomp}{\ensuremath{H'_\ell}}
\newcommand{\lhomx}[1]{\ensuremath{\lhom(#1)}}
\newcommand{\hH}{\ensuremath{\hat{H'_1}}}

\newcommand{\ec}[1]{\ensuremath{\overline{#1} } }
\newcommand{\wsp}[1]{\ensuremath{\llparenthesis #1 \rrparenthesis}}
\newcommand{\asp}[1]{\ensuremath{\left<#1\right>}}

\newcommand{\Cs}{Cauchy sequence}
\newcommand{\bp}{boundary point}
\newcommand{\njp}{non-stretching}

\newcommand{\finl}{\ensuremath{\sum_{e \in E(G)} \ell(e) < \infty}}

\newcommand{\vend}{vertex-end}
\newcommand{\eend}{edge-end}
\newcommand{\procir}{proper circle}

\newcommand{\clex}{circlex} 
\newcommand{\clexs}{circlexes} 
\newcommand{\rep}{representative}
\newcommand{\gms}{geodesic metric space}

\newcommand{\hycg}{\ensuremath{\ch (G)}}
\newcommand{\flcg}{\ensuremath{\ch_f (G)}}

\newcommand{\ksl}{Kirchhoff's second law}
\newcommand{\cutr}{non-elusive} 

\newcommand{\vol}{area}
\newcommand{\evol}{area}

\title{Graph topologies induced by edge lengths}

\author{Agelos Georgakopoulos\thanks{Supported by a grant of the German-Israeli Foundation and by an FWF grant. This work was conceived and written when the author was at the University of Hamburg.} \medskip \\
 {Technische Universit\"at Graz}\\
  {Steyrergasse 30, 8010}\\
  {Graz, Austria}\\
}

\begin{document}
\maketitle

\noindent

\begin{abstract}
Let $G$ be a graph each edge $e$ of which is given a length $\ell(e)$. This naturally induces a distance $d_\ell(x,y)$ between any two vertices $x,y$, and we let \ltp\ denote the completion of the corresponding metric space. It turns out that several well studied topologies on infinite graphs are special cases of $\ltp$. Moreover, it seems that $\ltp$ is the right setting for studying various problems. The aim of this paper is to introduce \ltp, providing basic facts,  motivating examples and open problems, and indicate possible applications. 

Parts of this work suggest interactions between graph theory and other fields, including algebraic topology and geometric group theory.

\end{abstract}

\section{Introduction}

Let $G$ be a graph each edge $e$ of which is given a length $\ell(e)$. This naturally induces a distance $d_\ell(x,y)$ between any two vertices $x,y$, and we let \defi{$\ltop(G)$}, or \defi{\ltp} for short, denote the completion of the corresponding metric space. It turns out that several well studied topologies on infinite graphs are special cases of $\ltop$, see \Sr{subSpec}. 

The space \ltp\ has already been considered, for special cases of $\ell$, by several authors who in most cases were apparently unaware of each other's work: Floyd \cite{floyd} used it in order to study Kleinian groups, and his work was taken up by Gromov who related it to hyperbolic graphs and groups,  see \Sr{secHyp}. Benjamini and Schramm  used it to prove that planar transient graphs admit harmonic Dirichlet functions \cite{BeSchrHar}, and to study sphere packings of graphs in $\R^d$ \cite{BeSchrLac}. Carlson \cite{CarBou} studied the Dirichlet Problem at the boundary of \ltp. Finally, the author used \ltp\ in \cite{AgCurrents} in order to prove the uniqueness of currents in certain electrical networks, see also \Sr{inele}.

The aim of this paper is to introduce \ltop\ in greater generality, providing definitions, motivating examples, basic facts and open problems, and indicate possible applications. 

\subsection{Interesting special cases of \ltop} \label{subSpec}

In \Sr{secSpe} we will show how some well known topologies on graphs can be obtained as special cases of \ltop\ by choosing $\lER$ appropriately. The most basic such example is the \FC\ (also known as the end compactification) of a \lf\ graph:

\begin{theorem} \label{finlf}
If $G$ is \lf\ and $\sum_{e \in E(G)} \ell(e) < \infty$ then $\ltp$ is homeo\-morphic to the \FC\ \fcg\ of \G.
\end{theorem}

Another special case of \ltop\ is the Floyd completion of a \lf\ graph, which in turn has as a special case the hyperbolic compactification of a hyperbolic graph in the sense of Gromov \cite{gromov}. Other special cases include the topologies $\etop$ and $\mtop$, which generalise the \FC\ to \nlf\ graphs, for all graphs for which these topologies are metrizable; see Sections~\ref{secDefs} and~\ref{secSpe} for definitions and the details.

\subsection{Infinite electrical networks} \label{inele}

Infinite electrical networks are a useful tool in mathematics, for example in the study of random walks \cite{LyonsBook}. An electrical network $N$ has an underlying graph \g and a function $r: E(G) \to \R_+$ assigning resistances to the edges of \G. If \g is finite, then the electrical current in $N$ ---between two fixed vertices $p,q$ and with fixed flow value $I$--- is the unique flow satisfying \ksl. If \g is infinite then there may be several flows satisfying \ksl, and one of the standard problems in the study of infinite electrical networks is to specify under what conditions such a flow is unique, see e.g.\ \cite{woessCurrents,thomInfNet}. 

In \cite{AgCurrents} we prove that if the sum of all resistances in a network $N$ is finite then there is a unique electrical current in $N$, provided we do not allow any flow to escape to infinity (see \cite{AgCurrents} for precise definitions):

\begin{theorem}[\cite{AgCurrents}] \label{finrI}
Let $N$ be an electrical network with $\sum_{e\in E(G)} r(e)<\infty$. Then there is a unique \cutr\ \flo{p}{q} with value $I$ and finite energy in $N$ that satisfies \ksl. 
\end{theorem}

The proof of \Tr{finrI} uses \Tr{finlf} and other basic facts about \ltop\ proved here (\Sr{basFacts}), as well as a result saying that, unless for some obvious obstructions, every flow satisfying \ksl\ for finite cycles in an infinite network also satisfies \ksl\ for infinite, topological circles in the space \ltp\ with $\ell:=r$.

\subsection{The cycle space of an infinite graph and the homology of a continuum} \label{insuccc}

The cycle space of a finite graph $G$ is the first simplicial homology group of $G$. This is a well studied object, and many useful results are known \cite{diestelBook05}. For infinite graphs many of these results fail even in the \lf\ case, however, \DK\ \cite{cyclesI,cyclesII} proposed a new homology for an infinite graph \G, called the \tcs\ \ccg, that makes those results true also for \lf\ graphs; an exposition of such results can be found in \cite{RDsBanffSurvey} or \cite[Chapter~8.5]{diestelBook05}. The main innovation of the approach of \DK\ was to consider topological circles in the \FC\ \fcg\ of the graph, and use those circles as the building blocks of their cycle space \ccg. 

It is natural to wonder how this homology theory interacts with \ltp: given an arbitrary function $\ell$ we can consider the topological circles in \ltp\ instead of those in \fcg, and we may ask if the former circles can be the building blocks for a homology that retains the desired properties of \ccg. 

The attempts to answer the latter question led to a new homology, introduced\sss\ in \cite{lhom}, that can be defined for an arbitrary metric space $X$ and has indeed important similarities to \ccg, at least if $X$ is compact. This homology is described in \Sr{sechom}, where we will also see an important example that motivates it.

\subsection{Geodetic circles} \label{ingeo}

As mentioned above, \fcg\ and the \tcs\ has led to generalisations of most well-known theorems about the cycle space of finite graphs to \lf\ ones, however, there are cases where \fcg\ performs poorly: namely, problems in which a notion of length is inherent. We will see one such case here; another can be found in \cite{AgCurrents}.

Let \g be a finite graph with edge lengths \lER. A cycle $C$ in $G$ is called \emph{$\ell$-geodetic} if, for any two vertices $x,y\in C$, the length of at least one of the two $x$--$y$~arcs on $C$ equals the distance between $x$ and $y$ in $G$, where lengths and distances are considered taking edge lengths into account. It is easy to show (see \cite{geo}) that:

\begin{theorem}\label{fingeo}
The cycle space of a finite graph \g is generated by its $\ell$-geodetic cycles.
\end{theorem} 

It was shown in \cite{geo} that this theorem generalises to \lf\ graphs using the \tcs, but only if the edge lengths respect the topology of \fcg, where  respecting the topology of \fcg\ means something slightly more general than  $\ltp$ being homeomorphic to $\fcg$. 

With \ltp\ we might be able to drop this restriction on $\ell$: we may ask whether for every \lER\ the \defi{$\ell$-geodetic} (topological) circles in \ltp\ ---defined similarly to finite $\ell$-geodetic cycles--- \defi{generate}, in a sense, all other circles; more precisely, we conjecture that they generate the homology group alluded to in \Sr{insuccc}. See \Sr{secgeo} for more. 


\subsection{Line graphs} \label{introLg}

The line graph $L(G)$ of a graph $G$ is defined to be the graph whose vertex set is the edge set of $G$ and in which two vertices are adjacent if they are incident as edges of $G$.

It is a well known fact that if a finite graph $G$ is eulerian then $L(G)$ is hamiltonian. This fact was generalised for locally finite graphs in \cite[Section~10]{fleisch}, where \et s and \fhcy s are defined topologically: a \defi{\hcy} is a homeomorphic image of $S^1$ in $\fcg$ containing all vertices, and a \defi{topological \et} is a continuous mapping from $S^1$ to $\fcg$ that traverses each edge precisely once.  

We would like to generalise this fact to \nlf\ graphs. In order to define \tet s and \hcy s as above, we first have to specify some topology for those graphs. Interestingly, it turns out that we can obtain an elegant generalisation to \nlfg s, but only if different topologies are used for $G$ and $L(G)$:

\begin{theorem}\label{Lg}
If $G$ is a countable graph and $\eti$ has a \tet\ then $\mtopx{L(G)}$ has a \hcy.
\end{theorem}
(See \Sr{defEnds} for the definition of $ETOP$ and \Sr{secNlf} for the definition of $MTOP$.) This phenomenon can however be explained using \ltop. If an assignment of edge lengths \lER\ is given for a graph $G$, then it naturally induces an assignment of edge lengths $\lL$ to the edges of $L(G)$: let $\lL(f)=1/2(\ell(e)+\ell(d))$ for every edge $f$ of $L(G)$ joining the edges $e,d$ of $E(G)$ (to see the motivation behind the definition of $\lL$, think of the vertices of $L(G)$ as being the midpoints of the edges of $G$). In \Sr{secSpe} we are going to show that if \finl, in which case $\ltp$ is homeomorphic to $\eti$, then $\ltopxl{\lL}{L(G)}$ is homeomorphic to $ \mtopx{L(G)}$, see \Cr{corLg}. It could be interesting to try to generalise \Tr{Lg} for other assignments $\ell$. 

We are going to prove \Tr{Lg} in \Sr{secLg}.

\section{Definitions and basic facts} \label{secDefs}

Unless otherwise stated, we will be using the terminology of Diestel \cite{diestelBook05} for graph theoretical terms and the terminology of \cite{armstrong} and \cite{Hatcher} for topological ones.

\subsection{Metric spaces} \label{secDefsM}

In this subsection we recall some well-known facts for the convenience of the reader. 
A function $f$ from a metric space $(X,d_x)$ to a metric space $(Y,d_y)$  is \defi{uniformly continuous} if for every $\eps>0$ there is a $\del>0$, such that for every $x,y \in X$ with $d_x(x,y)<\delta$ we have $d_y(f(x),f(y))<\eps$.
The following lemma follows easily from the definitions.

\begin{lemma} \label{uncon} 
Let $f:M \to N$, where $M,N$ are metric spaces, be a uniformly continuous function. If $(x_i)$ is a Cauchy sequence in $M$ then $(f(x_i))$ is a Cauchy sequence in $N$.
\end{lemma}

For every metric space $M$, it is possible to construct a complete metric space $M'$, called the \defi{completion} of $M$, which contains $M$ as a dense subspace. The completion  $M'$ of $M$ has the following universal property \cite{MiVoFun}:

\begin{equation} 
\label{univ} 
\begin{minipage}[c]{0.85\textwidth} 
If $N$ is a complete metric space and $f: M \to N$ is a uniformly continuous function, then there exists a unique uniformly continuous function $f': M' \to N$ which extends $f$. The space $M'$ is determined up to isometry by this property (and the fact that it is complete).
\end{minipage}\ignorespacesafterend 
\end{equation}

\comment{
	The following observation follows easily from the definition of equivalence of \Cs s.

	\begin{observation} \label{eqCau}
	For any metric space $M$, two \Cs s \seq{x}, \seq{y} in $M$ are equivalent if and only if there is a third \Cs\ in $M$ having infinitely many common elements with both $x_i$ and $y_i$.
	\end{observation}
}

A \defi{continuum} is a non-empty, compact, connected metric space.

\subsection{Topological paths, circles, etc.} \label{defsTop}

A \defi{circle} in a topological space $X$ is a homeomorphic copy of the unit circle $S^1$ of $\R^2$ in $X$.  An \defi{arc} $R$ in 
$X$ is a homeomorphic image of the real interval $[0, 1]$ in $X$. Its \defi{endpoints} are the images of $0$ and $1$ under any homeomorphism from \unin\ to $R$. If $x,y\in R$ then $xRy$ denotes the subarc of $R$ with endpoints $x,y$. A \defi{topological path} in $X$ is a continuous map from a closed real interval to  $X$. 

A \defi{\clex} is a singular 1-simplex traversing a circle $C$ once and in a straight manner; in other words, a continuous mapping $\sigma:\unin \to C$ that is injective on $(0,1)$ and satisfies $\sigma(0)=\sigma(1)$.

Let $\sigma: [a,b]\to X$ be a topological path in a metric space $(X,d)$. For a finite sequence $S=s_1, s_2, \ldots, s_k$ of points in $[a,b]$, let $\ell(S):= \sum_{1\leq i< k} d(\sigma(s_i),\sigma(s_{i+1}))$, and define the \defi{length} of $\sigma$ to be $\ell(\sigma):=\sup_S \ell(S)$, where the supremum is taken over 
all finite sequences $S=s_1, s_2, \ldots, s_k$ with $a=s_1<s_2< \ldots <s_k=b$. If $C$ is an arc or a circle in $(X,d)$, then we define its length $\ell(C)$ to be the length of a surjective topological path $\sigma:[0,1] \to C$ that is injective on $(0,1)$; it is easy to see that $\ell(C)$ does not depend on the choice of \sig.

\afsubm{removed now redundant Lemmas about eq. def of length}

We are going to need the following well-known facts. 

\LemArc{arc}

\LemInjHom{injHom}

\subsection{\ltop} \label{defsLtop}

Every graph \g in this paper is considered to be a 1-compex, which means that the edges of $G$ are homeomorphic copies of the real unit interval. A \defi{half-edge} of a graph $G$ is a connected subset of an edge of $G$.

Fix a graph $G$ and a function \lER, and for each edge $e\in E(G)$ fix an isomorphism $\sigma_e$ from $e$ to the real interval $[0,\ell(e)]$; by means of $\sigma_e$ any half-edge $f$ with endpoints $a,b$ obtains a length $\ell(f)$, namely $\ell(f):=|\sigma_e(a)-\sigma_e(b)|$. 

We can use $\ell$ to define a distance function on $G$: for any $x,y \in V(G)$ let $d_\ell(x,y)= \inf \{\ell(P) \mid P \text{ is an \pth{x}{y}}\} $, where $\ell(P):= \sum_{e \in E(P)} \ell(e)$. For points $x,y \in G$ that might lie in the interior of an edge we define $d_\ell(x,y)$ similarly, but instead of graph-theoretical paths we consider arcs in the 1-complex $G$: let 
$d_\ell(x,y)= \inf \{\ell(P) \mid P \text{ is an \arc{x}{y}}\}$, where $\ell(P)$ is now the sum of the lengths of the edges and maximal half-edges in $P$; note that this sum equals the length of $P$ as defined in \Sr{defsTop} (for metric spaces in general).
By identifying any two vertices $x,x'$ of $G$ for which $d_\ell(x,x')=0$ holds we obtain a metric space $(\gid{G}, d_\ell)$. Note that if \g is \lf\ then $\gid{G}=G$. Let $\ltp$ be the completion of $(\gid{G}, d_\ell)$.

The \defi{\bp s} of $G$ are the elements of the set $\blg:=\ltp \sm \pi(\gid{G})$, where $\pi$ is the canonical embedding of $\gid{G}$ in its completion \ltp.

For a subspace $X$ of $\ltp$ we write $E(X)$ for the set of edges contained in $X$, and we write $\kreis{E}(X)$ for the set of maximal half-edges contained in $X$; note that $\kreis{E}(X) \supseteq E(X)$. Similarly, for a topological path $\tau$ we write $E(\tau)$ for the set of edges contained in the image of $\tau$.

In this paper we will often encounter special cases of \ltp\ that induce some other well known topology on some space $G'$ containing $G$, e.g.\ the \FC\ of $G$. In order to be able to formally state the fact that the two topologies are the same, we introduce the following notation. Let $X,X'$ be  topological spaces that contain another topological space $G$. We will write $X \eqT_G X'$, or simply $X \eqT X'$ if $G$ is fixed, if the identity on $G$ extends to a homeomorphism between $G'$ and $G''$.

If \g is \lf\ then we could also define \ltp\ by the following definition, which is equivalent to the above as the interested reader will be able to check (using, perhaps, \Lr{lcomb} below). Let $R, L$ be two rays ---see \Sr{defEnds} below for the definition of a ray--- of finite total length in \G; we say that $R$ and $L$ are \defi{equivalent}, if there is a third ray of finite total length that meets both $R$ and $L$ infinitely often. Now define $\partial G$ to be the set of equivalence classes of rays  of finite total length in \G\ \wrt\ this equivalence relation, and let $\ltp:= G \cup \partial G$; to extend the metric $d_\ell$ from $G$ to all of \ltp, let the distance $d_\ell(x,y)$ between two points $x,y$ in $\partial G$ be the infimum of the lengths of all double rays with one ray in $x$ and one ray in $y$, and let $d_\ell(x,v)$ for $x\in  \partial G$ and $v\in G$ be the infimum of the lengths of all rays in $x$ starting at the point $v$.


\subsection{Ends, the \FC\ and the \tcs} \label{defEnds}

Let \g be a graph fixed thgoughout this section.

A $1$-way infinite path is
called a \defi{ray}, a $2$-way infinite path is
a \defi{double ray}. A \defi{tail} of the ray $R$ is a final subpath of $R$. 
Two rays $R,L$ in $G$ are \defi{vertex-equivalent} if no finite set of vertices
separates them; we denote this fact by $R\approx_G L$, or simply by $R\approx L$ if $G$ is fixed. The corresponding equivalence
classes of rays are the \defi{\vend\s} of $G$. We
denote the set of \vend\s\ of $G$ by \defi{$\Omega =
\Omega(G)$}. A ray belonging to the \vend\ \oo\ is an \defi{\oo-ray}. Similarly, two rays are \defi{edge-equivalent} if no finite set of edges separates them and we call the corresponding equivalence classes the \defi{\eend\s} of \G\ and let \defi{$\EOO(G)$} denote the set of \eend\s\ of \G. In a \lfg\ \g any two rays are edge-equivalent \iff\ they are vertex-equivalent, and we will simply call the  corresponding equivalence classes the \defi{ends} of \g.

We now endow the space consisting of \G, considered as a 1-complex, and its edge-ends with the topology \eti. Firstly, every edge $e\in E(G)$ inherits the open sets corresponding to open sets of $[0, 1]$. Moreover, for every finite edge-set $S \subset E(G)$, 
we declare all sets of the form 
\begin{equation}
C \cup \EOO(C) \cup
E'(C)
\label{eq}
\end{equation}
to be open, where $C$ is any component of $G - S$ and $\EOO(C)$ denotes the set of
all edge-ends of $G$ having a ray in $C$ and $E'(C)$ is any union of
half-edges $(z, y]$, one for every edge
$e = xy$ in $S$ with $y$ lying in $C$. Let \Getop\ denote the topological space of $G \cup
\EOO$ endowed with the topology generated by
the above open sets. Moreover, let \defi{\eti} denote the space obtained from \Getop\ by identifying any two points that have the same open neighbourhoods. If a point $x$ of \eti\ resulted from the identification of a vertex with some other points (possibly also vertices), then we will, with a slight abuse, still call $x$ a \defi{vertex}. It is easy to see that two vertices $v,w$ of \g are identified in \eti\ \iff\ there are \inm\ edge-disjoint \pths{v}{w}. 

If \g is \lf\ then $\Getop$ and \eti\ coincide, and it can be proved (see \cite{ends}) that \eti\ is the \defi{Freudenthal compactification}
\cite{Freudenthal31} of the 1-complex $G$ in that case. It is common in recent literature to denote this space by \defi{\fcg} if \g is \lf\footnote{\fcg\ is typically defined by considering vertex separators instead of edge separators and otherwise imitating our definition of \Getop\ (see e.g.\ \cite[Section~8.5]{diestelBook05}) but for a \lfg\ this does not make any difference.}, and we will comply with that convention.

The study of \fcg, in particular of topological circles therein, has been a very active field recently. It has been demonstrated by the work of several authors (\cite{locFinTutte,partition,bicycle,LocFinMacLane,degree,cyclesI,cyclesII,hp,fleisch,geo,hcs}) that many well known results about paths and cycles in finite graphs can be generalised to \lf\ ones if the classical concepts of path and cycle are interpreted topologically, i.e.\ replaced by the concepts of a (topological) arc and circle in \fcg; see \CDB\ or \cite{RDsBanffSurvey} for an exposition. An example of such a \tocir\ is formed by a double ray both rays of which converge to the same end, together with that end. There can however be much more exciting circles in \fcg: in \fig{wild}, the infinitely many thick double rays together with the continuum many ends of the graph combine to form a single \tocir\ $W$, the so-called \defi{wild} circle, discovered by \DK\ \cite{cyclesI}. The double rays are arranged within $W$ like the rational numbers within the reals: between any two there is a third one; see \cite{cyclesI} for a more precise description of $W$.

\showFig{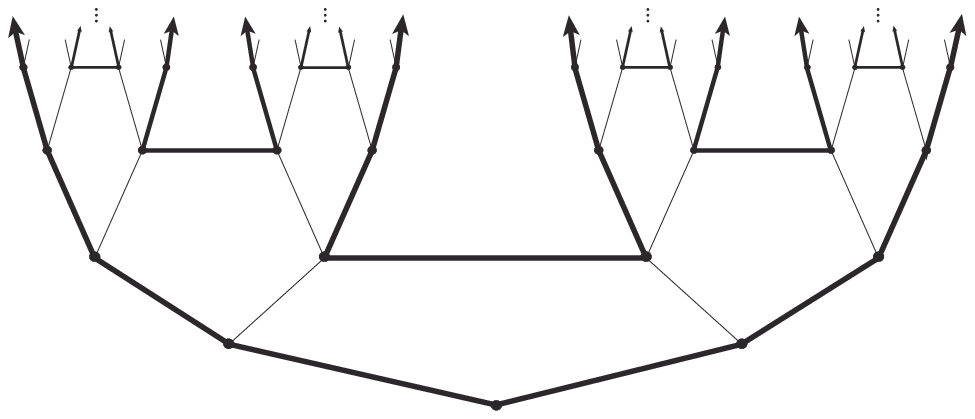}{The `wild' circle, formed by infinitely many (thick) double rays and con\-tinuum many ends. In this drawing it bounds the outer face of the graph.}

We finish this section with a basic fact about infinite graphs that we will use later. A \defi{comb} in \g is the union of a ray $R$ (called the \defi{spine} of the comb) with infinitely many disjoint finite paths having precisely their first vertex on $R$. A \defi{subdivided star} is the union of a (possibly infinite) set of finite paths that have precisely one vertex in common.

\LemCombStar{comb}


\section{Special cases of \ltop} \label{secSpe}

\subsection{The \FC\ and \eti} \label{secEtop}

We start this section by proving \Tr{finlf}, which states that the \FC\ of a \lfg\ is a special case of \ltp. Since for a \lfg\ \g the \FC\ coincides with the topology $\eti$, \Tr{finlf} is a corollary of the following.

\begin{theorem} \label{finl}
If $G$ is countable and $\sum_{e \in E(G)} l(e) < \infty$ then $\ltp \eqT \eti$.
\end{theorem}

\begin{proof}
The proof consists of two steps: in the first step we put a metric $d$ on $\eti$ similar to $d_\ell$ and show that this metric induces $\eti$, while in the second step we show that the corresponding metric space is the completion of $(\gid{G}, d_\ell)$ using property \eqref{univ} (the space $\gid{G}$ was introduced in \Sr{defsLtop}, and property \eqref{univ} in \Sr{secDefsM}). As the interested reader can check, it is also possible, and not harder, to prove \Tr{finl} by more direct arguments, without using \eqref{univ}.

For the first step, define $d(x,y):=d_\ell(x,y)$ for every $x,y$ in the 1-complex $G$. If $x \in \EOO(G)$ and $y \in V(G)$ (respectively $y \in \EOO(G)$), then let $d(x,y)$ be the infimum of the lengths of all rays in $x$ starting at $y$ (resp.\ all $x$--$y$ double rays), where the length of a (double-)ray $R$ is taken to be $\sum_{e\in E(R)}\ell(e)$. Define $d(x,y)$ similarly for the case that $x\in \EOO(G)$ and $y$ lies in an edge. It is easy to check that $d$ is indeed a metric. We claim that $d$ induces the topology $\eti$. To prove this we need to show that for any open set $O$ of $\eti$ and any $x \in O$ there is a ball $B\ni x$ with respect to $d$ contained in $O$ and vice versa. 

So suppose firstly that $O$ is a basic open set in $\eti$ with respect to the finite edge-set $F$, and pick an $x \in O$. If $x$ is an inner point of some edge $f\in F$, then it is easy to find a ball of $x$ contained in $f\cap O$, so we may assume that $x$ is not such a point. Let $r=\min_{e \in F} \ell(e)$. Then, the ball $B_d(x, r)$ is contained in $O \cup \bigcup F$, since for any point $y$ in $G$ that $F$ separates from $x$ we have $d(x,y)\geq r$ by the definition of $d$. Thus, easily, there is an $r'\leq r$, depending on $O \cap \bigcup F $, such that $B_d(x, r') \subseteq O$.

Next, pick a ball $U=B_d(x,r)$ and a $y \in U$.
We want to find an open set $O$ in $\eti$ such that $y \in O \subseteq U$. Easily, we can again assume that $y$ is not an inner point of an edge. Moreover, we may assume \obda\ that $x=y$. As $\sum_{e \in E(G)} l(e) < \infty$ holds, there is a finite edge set $F$ such that the sum $\sum_{e \in E(G) \sm F} \ell(e)$ of the lengths of all edges not in $F$ is less than $r$. We claim that any basic open set $O$ of \eti\ \wrt\ the edge-set $F$ that contains $x$ is a subset of $U$. Indeed, for any point $w \in O$ we have $d(x,w) <r$ because the sum of the lengths of all edges in $O$, and thus also in any path or ray in $O$, is less than $r$.

Thus $d$ induces $\eti$ as claimed. It is easy to check that $d$ is a pseudometric  on the set of points $M$ in $\eti$. To see that it is a metric, note that for any two points $w,z$ in $M$ that can be separated by a finite edge set $F$ we have $d(w,z)\geq \min_{e\in F} \ell(e)>0$. The second step of our proof is to show that the metric space $(M,d)$ is isometric to the completion of $(\gid{G}, d_\ell)$, that is, \ltp, and we will do so using \eqref{univ}.

We first need to show that $(M,d)$ is complete. To do so, let $(x'_i)$ be a \Cs\ in $(M,d)$. If there is an edge containing \inm\ of the $x'_i$ then it is easy to see that $(x'_i)$ has a limit, so assume this is not the case. Then, as \finl, it is possible to replace every $x'_i$ that is an inner point of an edge by a vertex $x_i$ close enough to $x'_i$, to obtain a sequence of vertices $(x_i)$ equivalent to $(x'_i)$. If the set $\{x_i \mid \iin\}$ is finite then one of its elements is a limit of $(x'_i)$. If it is infinite, then by \Lr{comb} there is either a comb with all teeth in $(x_i)$ or a subdivision of an infinite star with all leaves in $(v_i)$. If the former is the case, then $(x_i)$, and thus $(x'_i)$, converges to the limit of the comb, and if the latter is the case then both sequences converge to the center of the star. This proves that $(M,d)$ is complete.

It is easy to see that two vertices $v,w$ of \g are identified in \eti\ \iff\ there are \inm\ edge-disjoint \pths{v}{w}, which is the case \iff\ $v$ and $w$ are identified in $\gid{G}$. Thus we can define a canonical projection $\pi: \gid{G} \to \eti$, mapping an inner point of an edge to itself, and mapping an equivalence class in $\gid{G}$ of vertices of $G$ to the element of $M$ containing all these vertices. It is straightforward to check that $\pi$ is an isometry and that its image is dense in \eti.

In order to use  \eqref{univ}, let $(X,d_x)$ be a complete metric space, and let $f: \gid{G} \to X$ be a uniformly continuous function. In order to extend $f \circ \pi^{-1}: M \to X$ into a uniformly continuous function $f': M \to X$, given $\oo \in \EOO(G) \cap \eti$ pick an \oo-ray $R$, and let $v_1, v_2, \ldots $ be the sequence of vertices in $R$. As $\sum_{e \in R} \ell(e) < 0$ it is easy to check that $(v_i)$ is a \Cs, and thus by \Lr{uncon} $(f(v_i))$ is a \Cs\ in $X$. Let $x$ by the limit of $(f(v_i))$ in $X$ and put $f'(\oo)=x$. It is now straightforward to check that $f'$ is uniformly continuous. Moreover, as for any $\oo \in \EOO(G)\cap \eti$ there are vertices $x_i$ (e.g.\ the vertices of an \oo-ray) such that $d(\oo, x_i)$ becomes arbitrarily small, it is easy to check that this $f'$ is the only (uniformly) continuous extension of $f\circ \pi^{-1}$. Thus, by \eqref{univ}, $(M,d)$ is isometric to the completion of $(\gid{G}, d_\ell)$, that is, $\ltp$.
\end{proof}

\Tr{finl}\ plays an important role in the proof of \Tr{finrI}. A further application is an easy proof of the following known fact (see \cite[Proposition~8.5.1]{diestelBook05} for the locally finite case or \cite[Section~2.1]{schulz} for the general case).

\begin{corollary} \label{corcomp}
If \g is a connected countable graph then \eti\ is compact.
\end{corollary} 
\begin{proof}
It is not hard to see that if an assignment \lER\ satisfies \finl\ then \ltp\ is totally bounded. The assertion now follows from \Tr{finl} and the fact that every complete totally bounded metric space is compact.
\end{proof}

It is natural to wonder whether the condition $\finl$ in \Tr{finl} can be replaced by some weaker but still elegant condition. However, this does not seem to be possible, as indicated by the following example: \fig{monsterltop} shows an 1-ended \lfg\ \g and an assignment of lengths \st\ \fe\ \eps\ there are only finitely many edges longer than \eps. Still, as the interested reader can check, \ltp\ does not induce \fcg\ in this case. Even worse, as we move along the bottom horizontal ray of this graph, the distance to the limit of the upper horizontal ray grows larger, although the two rays converge to the same point in \fcg.

\showFig{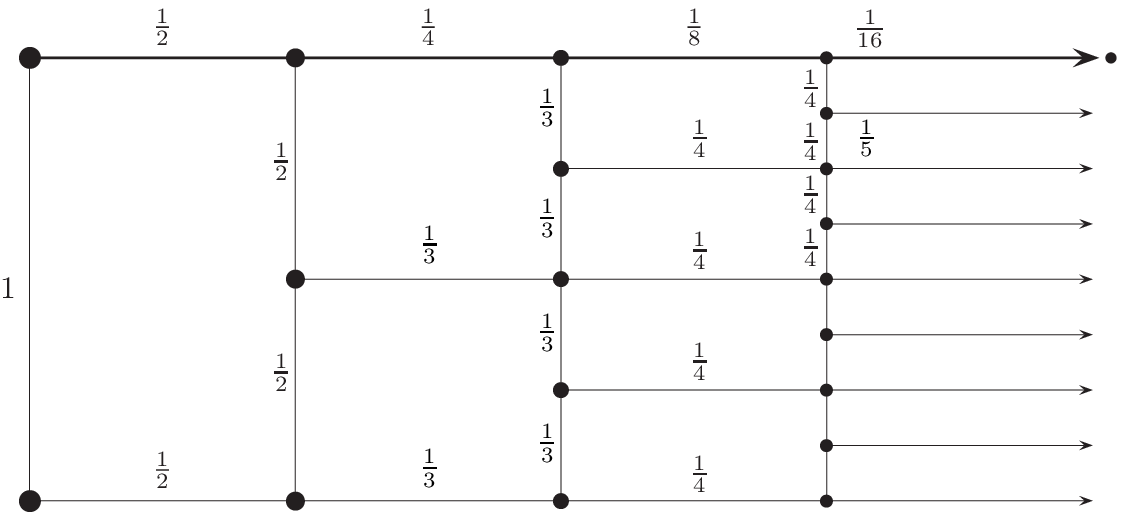}{A somewhat surprising example.}

\subsection{The hyperbolic compactification} \label{secHyp}

In a seminal paper \cite{gromov} Gromov introduced the notions of a hyperbolic graph and a hyperbolic group, and defined for each such graph a compactification  \hycg, called the \defi{hyperbolic compactification} of \G, that refines \fcg. It turns out that this space \hycg\ is also a special case of \ltp. 

For the definitions of hyperbolic graphs and their compactification the interested reader is referred to \cite{Short}. An example of a hyperbolic graph \g is shown in \fig{exhyp}. Other examples include all tessellations of the hyperbolic plane. The study of finitely generated groups whose Cayley graphs are hyperbolic is a very active research field with many applications, see \cite{kapSurv} for a survey.

Another notion related to the hyperbolic compactification is that of the \defi{Floyd completion}. To define it, let \g be a \lfg\ and let $f: \N \to \R_{>0}$ be a summable function, i.e.\ $\sum_\nin f(n) <\infty$, that does not decrease faster than exponentially; formally, there is a constant $\lambda>0$ such that $\lambda f(n-1)\leq f(n) \leq f(n-1)$ for every $n>0$. Now fix a vertex $p$ of \g and assign to each edge $e\in E(G)$ the length $\ell(e):= f(d(p,e))$, where $d(p,e)$ denotes the graph-theoretical distance, i.e.\ the least number of edges that form a path from $p$ to one of the endvertices of $e$. We define the \defi{Floyd completion} \flcg\ of \g (with respect to $f$) to be \ltp\ for this $\ell$. Floyd introduced this space in \cite{floyd} and used it in order to study Kleinian groups\footnote{Floyd did of course not use the term \ltop; but he defined \flcg\ the same way as we do.}. Gromov showed (\cite[Corollary~7.2.M]{gromov}) that if \g is hyperbolic then $f$ can be chosen in such a way (in addition to the above properties) that the Floyd completion coincides with the hyperbolic compactification (see \cite{CDP} for a more detailed exposition):

\begin{theorem}[\cite{gromov}] \label{gromov}
For every \lf\ hyperbolic graph \g there is a constant $\eps\in \R_+$ such that the hyperbolic compactification \hycg\ of \g is homeomorphic to its Floyd completion \flcg\ for $f(n):=\exp({-\eps n})$.
\end{theorem}

Since the Floyd completion is explicitly defined as a special case of \ltp, we immeadiately obtain

\begin{corollary}\label{}
For every \lf\ hyperbolic graph \g there is \lER\ \st\ the hyperbolic compactification \hycg\ of \g is homeomorphic to \ltp.
\end{corollary} 

\showFig{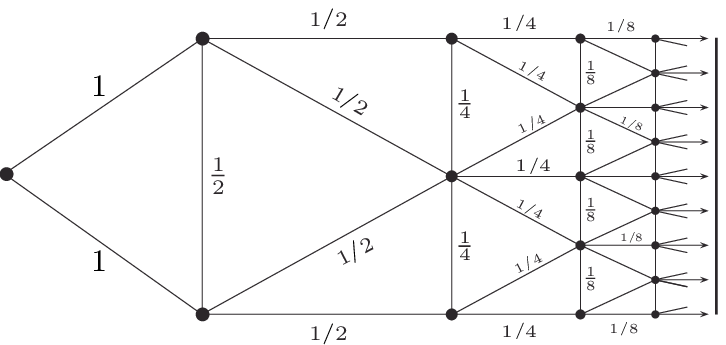}{A hyperbolic graph $G$. Each \defi{level} $i$ of $G$ is a perpendicular path $P_i$ of  $2^i$ edges from the upper to the bottom horizontal ray. 
The \defi{hyperbolic boundary} of this graph, i.e.\ $\hycg \sm G$, is homeomorphic to a closed real interval, and any two disjoint horizontal rays converge to distinct points in this boundary.}

In the graph of \fig{exhyp} for example, if we let $f(n):= 2^{-n}$ then \flcg\ will be homeomorphic to \hycg. (Note however that not any exponentially decreasing $f$ would do; if we let $f(n):= 4^{-n}$ for instance, then \flcg\ will be homeomorphic to \fcg.)

Intuitively, hyperbolic graphs are characterised by the property that for any two geodetic rays $R,L$ starting at a vertex $x$ of the graph graph, one of two possibilities must occur: either there is a constant $C$ such that each vertex of $R$ is at most $C$ edges apart from some vertex of $L$, or $R$ and $L$ diverge exponentially; that is, the minimum length of an \pth{R}{L}\ $P_r$ outside the ball of radius $r$ around $x$ grows exponentially with $r$. (See \cite[Definition~1.7]{Short} for a more precise statement of this fact.) The function $f$ in \Tr{gromov} is chosen in such a way, that for any two rays $R,L$ as above that diverge exponentially the paths $P_r$ are assigned lengths that are bounded away from 0, and thus $R$ and $L$ will converge to distinct points in \flcg.

\subsection{Non-locally-finite graphs} \label{secNlf}

The topology \eti\ we used in \Sr{secEtop} compactifies \g by its {edge-ends}. 
A further popular \cite{diestelESST} possibility to extend \fcg\ to a \nlfg\ \g is the topology $\mtopx{G}$, which consists of \g and its \vend\s. As we shall see,  $\mtopx{G}$ is also a special case of \ltp. To define $\mtopx{G}$ we consider each edge of \g to be a  copy of the real interval \unin, bearing the corresponding metric and topology. The basic open neighbourhoods of a vertex $v$ are, then, taken to be the open stars of radius \eps\ centered at $v$ for any $\eps<1$. For a \vend\ $\oo\in \OO(G)$ we declare all sets of the form 
$$\widehat{C_\eps}(S,\oo) := C(S,\oo) \cup \OO(S,\oo) \cup \kreis{E}_\eps(S,\oo)$$
to be open, where $S$ is an arbitrary finite subset of $V(G)$, $C(S,\oo)$ is the unique component of $G - S$ containing a ray in \oo, and $\kreis{E}_\eps(S,\oo)$ is the set of all inner points of $S-C(S, \oo)$ edges at distance less than $\eps$ from their endpoint in $C(S,\oo)$.

As stated in \Sr{introLg}, if $G$ is a countable graph and \finl\ for some \lER, then $\ltopxl{\lL}{L(G)} \eqT \mtopx{L(G)}$; this is implied by the following statement, which can be proved by imitating the proof of \Tr{finl}. 
	\note{don't remove mention of  \Sr{introLg} here, because what you are proving next is a special case of the result after that}

\begin{corollary}\label{corLg}
Let $G$ be a countable graph, $v_1,v_2,\ldots$ an enumeration of the vertices of $G$, and for every edge $v_i v_j$ let $\ell'(v_i v_j)=1/2(\ell_i + \ell_j)$ where $\seq{\ell}$ is a sequence of positive real numbers with $\sum_i \ell_i <\infty$. Then $\ltopxl{\ell'}{G} \eqT \mtopx{G}$.
\end{corollary} 

This shows that $\mtopx{G}$ is a special case of \ltp\ if $G$ is countable, but in fact we can do better and drop the latter requirement as long as $\mtopx{G}$ is metrizable. The graphs $G$ for which this is the case were characterised by Diestel:

\begin{theorem}[\cite{diestelESST}] \label{mtop}
If $G$ is connected then $\mtp$ is metrizable if and only if $G$ has a normal spanning tree.
\end{theorem}

We can show that $\mtopx{G}$ is a special case of \ltp\ for all those graphs:

\begin{theorem}[{\cite[Theorem 3.1]{diestelESST}}] \label{mltop}
Let $G$ be a connected graph. Then, there is $\lER$ such that $\ltp \eqT \mtp$ if and only if $\mtp$ is metrizable.
\end{theorem}

\begin{proof}[Proof (sketch).]
The forward implication is trivial. For the backward implication we will proceed similarly to the proof of \Tr{finl}. So suppose that $\mtp$ is metrizable.  Then, by \Tr{mtop}, $G$ has a normal spanning tree $T$. For a vertex $v \in V(G)$ let $r(e)$ be the level of $v$ in $T$, that is, the number of edges in the path in $T$ from the root of $T$ to $v$. We now specify the required edge lengths $\ell$: for every edge $e=uv$, where $r(u)< r(v)$, let $\ell(e)= \sum_{r(u)<n\leq r(v)} 1/2^n$. Now define a metric $d$ on the point set of \mtp\ as follows. For every $x,y \in V(G) \cup \OO(G)$, let $d(x,y)= \sum_{e \in P} \ell(e)$, where $P$ is the \pth{x}{y} in $T$ if both $x,y$ are vertices, and $P$ is the $x$-$y$ (double-)ray in $T$ if one or both of $x,y$ is a \vend. Define $d(x,y)$ for inner points of edges similarly. Clearly, $d(x,y)=d_\ell(x,y)$ if $x,y \in G$. It is straightforward to check that $d$ induces $\mtp$; see \cite[Theorem 3.1]{diestelESST} for more details. Now similarly to the proof of \Tr{finl} one can show that $\mtp$ is the completion of $(\gid{G}, d_\ell)$ using \eqref{univ}.
\end{proof}

A further topology for infinite (\nlf) graphs that might be obtainable as a special case of \ltop\ is the topology of \defi{metric ends}. These are defined similarly to \vend s, only the role of finite vertex separators is now played by separators of finite diameter with respect to the usual edge-counting metric; see \cite{KroenEnds,KrMoMet} for more details. We can thus ask:

\begin{problem}
Does every graph \g admit an assignment \lER\ \st\ the identity on \g induces a bijection between \blg\ and the set of metric ends of \G? If yes, are the corresponding topologies homeomorphic?\footnote{I would like to thank the anonymous referee for proposing this problem.}
\end{problem}

The end compactification \fcg\ of a \lf\ graph \g has allowed the generalisation of many important facts about finite graphs to infinite, \lf\ ones, see \cite{cyclesIntro}. When trying to extend those results further to \nlf\ graphs however, one often has to face the dilemma of which topology to use, as there are several ways to generalise \fcg\ to a \nlf\ graph. In this section we considered two of these ways, namely the spaces \eti\ and \mtp, and ``unified'' them by showing that they are both special cases of \ltop\ (\Tr{finl} and \Tr{mltop}). This unification suggests a solution to the aforementioned dilemma: instead of fixing a topology on the \nlf\ graph, one could try to prove the desired result for all instances of \ltop, or at least for a large subclass of them like e.g.\ the compact ones, which would then lead to corollaries for the specific spaces. This approach will be exemplified in \Sr{sechomi}.

\section{Basic facts about \ltp} \label{basFacts}

Let $G$ be a graph and let \lER\ be fixed throughout this section.

\begin{lemma}	\label{lcomb}
Let \seq{x}, with $x_i\in V(G)$, be a \Cs\ in \ltp. Then, $G$ has a subgraph $R$ such that $R$ is either a comb or a subdivided star, $R$ contains infinitely many vertices in $\{x_i\}$, and $V(R)$ converges to the limit of \seq{x}.
\end{lemma}
\begin{proof}
Fix an $\eps\in \R$, $\eps>0$, and for every $i\geq 1$ let $R_i$ be a finite \pth{x_i}{x_{i-1}}\ with $\ell(R_i)< d_\ell(x_i,x_{i-1}) + \eps 2^{-i}$; such a path exists \btdo\ $d_\ell$. The subgraph $R':=\bigcup R_i$ of $G$ is connected, and applying \Lr{comb} to $\{x_i |\iin\}$ and $R'$ yields a subgraph $R$ of $R'$ which is either a comb or a subdivided star and contains infinitely many vertices in $\{x_i\}$. To see that $V(R)$ converges to the limit of \seq{x}, note that for any vertex $v\in R_i$ we have $d_\ell(v,x_i)< d_\ell(x_i,x_{i-1}) + \eps 2^{-i}$ and recall that \seq{x} is a \Cs.
\end{proof}


An $x$-$y$~geodesic in a metric space $X$ is a map $\tau$ from a closed interval $[0,l]\subset \R$ to $X$ such that $\tau(0)=x$, $\tau(l)=y$, and $d(\tau(t),\tau(t'))=|t-t'|$ for all $t,t' \in [0,l]$. If there is an $x$-$y$~geodesic \fe\ two points $x,y\in X$, then we call $X$ a \defi{\gms}.

In general, \ltp\ is not a \gms\ as shown by the graph in \fig{ladder}. In this space, the two \bp s $x,y$ have distance $3$, but there is no $\arc{x}{y}$ of length $3$. This example can be modified to obtain a space in which there are two vertices connected by no geodesic. 

\showFig{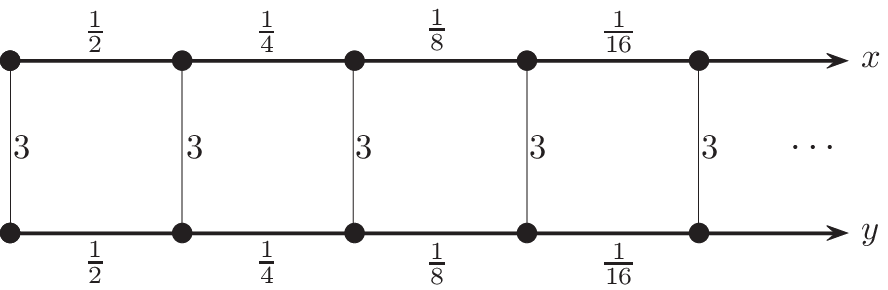}{In this example there is no geodesic connecting the \bp s $x,y$ in \ltp.}

However, \ltp\ is always a \gms\ if it is compact:

\begin{theorem}\label{thgeo}
If \ltp\ is compact then it is a \gms.
\end{theorem}
\begin{proof}
It is an easy and well-known fact that a complete metric space is geodesic if it has ``midpoints'', that is, if for any two points $x,y$ in the space there is a point $z$ so that $d(x,z)=d(z,y)=\frac{1}{2}d(x,y)$ holds. Let us show that \ltp\ does have midpoints if it is compact.

Pick $x,y\in \ltp$, and let $h:=d_\ell(x,y)$. Choose a sequence \seq{P}\ of finite paths in \g such that if $x_i, y_i$ are the endvertices of $P_i$ then the sequence \seq{x}\ converges to $x$,  the sequence \seq{y}\ converges to $y$, and $\lim \ell(P_i)=h$; such a sequence \seq{P}\ exists by the definition of \ltp. Now for every $i$, consider $P_i$ as a topological path in the 1-complex \G, and let $p_i$ be the midpoint of $P_i$, that is, the point on $P_i$ satisfying $ \ell(x_i P_i p_i) = \ell(p_i P_i y_i)$. As by our assumption \ltp\ is compact, the sequence \seq{p}\ has an accumulation point $z$, and it is easy to check that $d_\ell(x,z)=d_\ell(z,y)=\frac{1}{2}d(x,y)$ as desired.
\end{proof}

Next, we are going to prove two results that are needed in \cite{AgCurrents} but might be of independent interest. 

For the following two lemmas suppose \g is countable, fix an enumeration $e_1,e_2,\ldots$ of $E(G)$, and let $E_n:=\{e_1,\ldots, e_n\}$. Moreover, let $\kreis{e_n}$ denote the set of inner points of the edge $e_n$, and let $\kreis{E}_n:= \{\kreis{e_1},\ldots, \kreis{e_n}\}$.

\begin{lemma} \label{epsNLF} 
Let $C$ be a circle or arc in \ltp\ such that $E(C)$ is dense in $C$. Then, for every $\epsilon\in \R_+$ there is an $n\in\N$ such that for every subarc of $C$ in $\ltpn$ connecting two vertices $v,w$ we have  $d_\ell(v,w)<\epsilon$.
\end{lemma}
\begin{proof}

We will only consider the case when $C$ is an arc; the case when $C$ is a circle is similar. Suppose, on the contrary, there is an $\epsilon$ such that for every $n\in\N$ there is a subarc $R_n$ of $C$ in \ltpn\ 
the endvertices of which have distance at least $\epsilon$. Let $x_n,y_n$ be the first and last vertex of $R_n$ respectively along $C$. 

As $C$ is compact, the sequence \seq{x} has a convergent subsequence \sseq{x}{I} with limit $x$ say, and the corresponding subsequence \sseq{y}{I} of \seq{y} also has a converging subsequence \sseq{y}{J}, $J\subseteq I$, with limit $y$ say. Since $d_\ell(x_i,y_i)\geq \eps$ we have $d_\ell(x,y)\geq \eps$, in particular $x\neq y$. Note that $C$ must contain the points $x,y$ because it is closed. But as $E(C)$ is dense in $C$, $xCy$ must meet some edge $e_n$ at an inner point $z$ say. It is easy to see that if $m$ is large enough then $x_m$ and $y_m$ lie in distinct components of $C\sm \{z\}$. Thus for such an $m$ the subarc $R_m=x_m C y_m$ contains $z$, and as we can choose $m$ to be larger than $n$ this contradicts the fact that $R_m$ avoids $\kreis{e_n}$.
\end{proof}


\begin{lemma}\label{lcisle}
If $\sum_{e\in E(G)} \ell(e)<\infty$ then for every circle or arc $C$ in \ltp\ we have  $\ell(C)=\sum_{e\in E(C)} \ell(e)$.
\end{lemma} 

Before proving this, let us see an example showing that the requirement \finl\ in \Lr{lcisle} is necessary, since without it $\ell(C)$ might be strictly greater than $\sum_{e\in E( C)} \ell(e)$ even if $E(C)$ is dense in $C$. In fact, this situation can occur even if $\ltp \eqT \eti$ ---which is always the case if \finl\ by \Lr{finl}. The existence of a circle $C$ in \ltp\ such that 
$E(C)$ is dense in $C$ but still $\ell(C) > \sum_{e\in E( C)} \ell(e)$ might look surprising at first glance, but perhaps less so bearing in mind that it is possible to construct a dense set of non-trivial subintervals of the real unit interval of arbitrarily small total length (e.g.\ imitating the construction of the Cantor set but removing shorter intervals).

\example{ant}{
\normalfont
Let \g be the graph of \fig{antares}, which is the same as the graph of \fig{wild}, and let every thin edge in the $i$-th level have length $c2^{-i}$ for some fixed $c\in \R_{>0}$ (in this example we chose $c=1$). Moreover, assign lengths to the thick edges in such a way that the sum $s$ of the lengths of all thick edges 
is finite ($s=1\half$ in this example). Here we have, as an exception, allowed some edges to have length 0, but this can be avoided by contracting those edges. It is not hard to check that for this assignment $\ell$ we have $\ltp \eqT \fcg$. Furthermore, it is not hard to check that the length of the wild circle $W$ is (at least) $s + c$, which is greater than $\sum_{e\in E(W)} \ell(e)=s$ (in this example, the length of the outer double ray $L$ is 1 and the length of $W\sm L$ is $1\half$). To see this, note for example that the two ends of the outer double ray $L$ have distance $\min(c,\ell(L))$, since every finite path $P$ connecting the leftmost ray to the rightmost one has length at least $c$ (use induction on the highest level that $P$ meets). 

\showFig{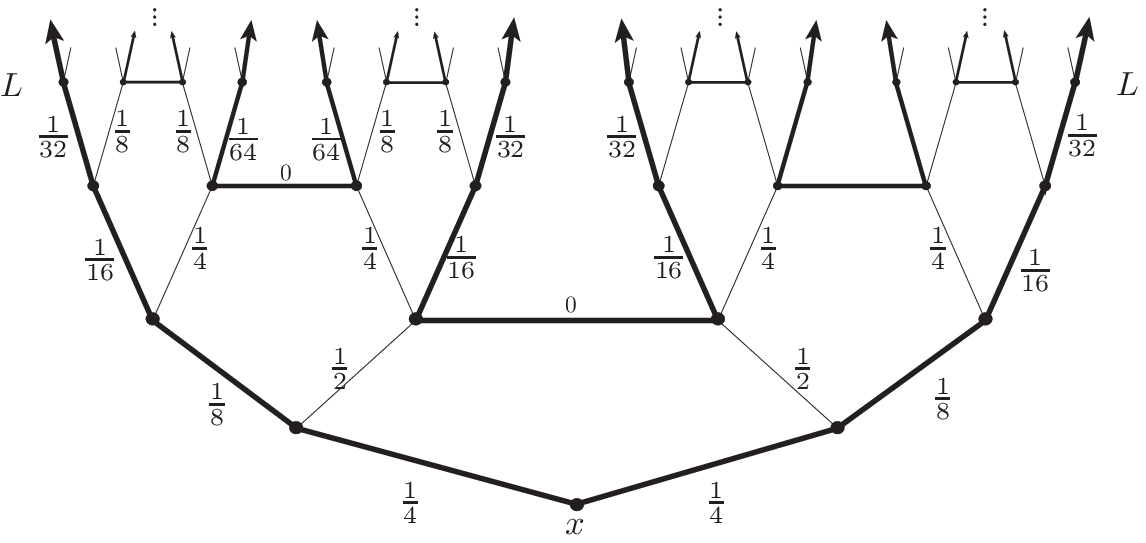}{The length of the wild circle $W$ here is greater than the sum of the lengths of the edges in it.}

} 

We now proceed to the proof of \Lr{lcisle}.

\begin{proof}[Proof of \Lr{lcisle}]
By \Tr{finl} we have $\ltp \eqT \eti$, and this easily implies that $E(C)$ is dense in $C$. \afsubm{simplified}
By the definition of $\ell(C)$ it suffices to show that 
\labtequ{xcyle}{$d_\ell(x,y)\leq \sum_{e \in E( xCy)} \ell(e)$}
holds for every pair $x,y$ of points on $C$. 
To show that \eqref{xcyle} holds, we will construct a sequence $\seq{P}$ of finite paths in $G$ such that the endpoints $x_i, y_i$ of $P_i$ give rise to sequences converging to $x,y$ respectively, and for the sequence $\delta_i:=\sum_{e\in E(P_i)\setminus E(xCy)} \ell(e)$ we have $\lim \delta_i=0$. This would easily imply $d_\ell(x,y)\leq \sum_{e\in E(xCy)} \ell(e)$ by the definition of $d_\ell$.

To begin with, pick an $r_0\in \N$ and let $P_0$ be any path in $G$. Then, for every $i=1,2,\ldots$, pick an $r_i>r_{i-1}$ large enough that the 
distance (in \ltp) between the endvertices of any subarc of $xCy$ in $\ltpr{r_i}$ is less than the length $h_{i-1}$  of the shortest edge in $\bigcup_{j<i}  P_{j}$; such an $r_i$ exists by \Lr{epsNLF}.

In order to define $P_i$, let $x_i$ be the first vertex of $xCy$ incident with $E_{r_i}$ and let $y_i$ be the last vertex of $xCy$ incident with $E_{r_i}$; such vertices exist because $xCy$ is a topological path and there are only finitely many vertices incident with $E_{r_i}$. Since $E(C)$ is dense in $C$, and since $\lim r_i=\infty$, it is easy to see that $x_i$ converges to $x$ and $y_i$ converges to $y$. Now for every maximal subarc $A$ of $x_iCy_i$ in $\ltpr{r_i}$ choose a finite \pth{w}{v} $P_A$, where $w,v$ are the  endpoints of $A$, 
\st\ $\ell(P_A) < h_{i-1}$; such a path $P_A$ exists since by the choice of $r_i$ we have $d_\ell(w,v) < h_{i-1}$. Then, replace $A$ in $x_iCy_i$ by the path $P_A$. Note that there are only finitely many such arcs $A$ as $E_{r_i}$ is finite. Performing this replacement for every such subarc $A$ we obtain a (finite) $x_i$-$y_i$~walk $P'_i$; let $P_i$ be an $x_i$-$y_i$~path with edges in $P'_i$.

By the choice of the paths $P_A$ we obtain that the edge-sets $E(P_i)\setminus E(xCy)$ are pairwise disjoint, 
and as $\sum_{e\in E(G)} \ell(e)<\infty$ this implies  $\lim \delta_i=0$ as required.
\end{proof}

\note{
	\Lr{lcisle} can be generalised to topological paths. The situation is slightly more complicated in this case though, since the length of a topological path does not only depend on its image, but increases whenever the path traverses some part of its image `back and forth'. Thus, in order to state the generalisation, we introduce the following concept that counts how often a given edge is traversed by a topological path. 

For a topological path $\sigma$ and an edge $e$, let $\ci$ be the set of maximal subpaths $I$ of $\sigma$ such that the image of $I$ is contained in $e$. Then, define the \defi{multiplicity} $m_\sigma(e)$ of $e$ in $\sigma$ by $m_\sigma(e):=\sum_{I\in \ci} \ell(I)$.

\begin{corollary}\label{lpisle}
If $\sum_{e\in E(G)} \ell(e)<\infty$ then for every topological path $\sigma:\unin\to\ltp$ such that $E(\sigma)$ is dense in the image of any subpath of $\sigma$ we have $\ell(\sigma)=\sum_{e\in E(G)} m_\sigma(e)$.
\end{corollary} 
\begin{proof} (sketch)
Let $x,y$ be any two points in \unin. By \Lr{arc}, the image of $x\sigma y$ contains an \arc{x}{y} $R$, and by \Lr{lcisle} $\ell(E(R))\geq \ell(R) \geq \ell(x\sigma y)$. The result now follows from \Lr{lengthPath} as in the proof of \Lr{lcisle}.
\end{proof}
}

For $\oo\in \OO(G)$ denote by \defi{\blom} the set 
$$\{x \in \blg \mid \text{ there is an \oo-ray that converges to $x$ in \ltp} \}.$$
(Recall that $\blg$ is the set of boundary points of \ltp.)
For the rest of this section assume \g to be \lf. Note that by \Lr{lcomb} every point in \blg\ lies in \blom\ for some end \oo. It has been proved that if \g is a hyperbolic graph and \ltp\ is its hyperbolic compactification (see \Sr{secHyp}), then \blom\ is a connected subspace of \ltp\ for every $\oo\in\OO(G)$ \cite[Chapter 7 Proposition 17]{GhHaSur}. The following proposition generalises this fact to an arbitrary compact \ltp.

\begin{theorem}\label{connbound}
If \g is \lf\ and \ltp\ is compact then \blom\ is connected for every $\oo\in\OO$.
\end{theorem} 

Before proving this, let us make a couple of related remarks. If $\oo, \psi$ are distinct ends of \G, then there is a finite edge set $S$ that separates them. But then, \fe\ assignment \lER, and any choice of elements $x\in \blom, y\in  \bxx{\ell}{\psi}$, we have $d_\ell(x,y) \geq r$ where $r$ is the minimum length of an edge in $S$. This implies, firstly, that \blom\ is a closed subspace of \blg, and thus of \ltp, \fe\ end $\oo$, and secondly, that every two points $x,y$ as above lie in distinct components of \blg. Thus, in the case that \ltp\ is compact, \Tr{connbound} characterises the components of \blg: they are precisely the sets of the form \blom.

\begin{proof}[Proof of \Tr{connbound}.]
Suppose, on the contrary, that \ltp\ is compact but \blom\ disconnected. Then,
there is a bipartition  $\{O_1,O_2\}$ of \blom\ where both $O_1,O_2$ meet \blom\ and both are clopen in the subspace topology of \blom. 
As \blom\ is a closed subspace of \ltp\ it must be compact, and so there is a lower bound $r>0$ \st\ \fe\ $x\in O_1$ and $y\in O_2$ we have $d_\ell(x,y)> r$. 
Let $U_1:= \bigcup_{x\in O_1} B_x(\frac{r}{2})$ and $U_2:= \bigcup_{y\in O_2} B_y(\frac{r}{2})$ where the sets of the form $B_x(\frac{r}{2})$ are open balls in \ltp; note that $U_1,U_2$ are disjoint open sets of \ltp.

Pick a point $x \in O_1$ and a point $y \in O_2$. Moreover, pick \oo-rays $R,T$ in $G$ converging (in \ltp) to $x,y$ respectively. 
Since $R$ and $T$ are vertex-equivalent, there is an infinite set $\mathcal P$ of pairwise disjoint \pths{R}{T}. 

Note that $U_1$ contains a tail of $R$ and $U_2$ contains a tail of $T$. If $U_1 \cup U_2$ contains \inm\ of the paths in $\cp$, then we can combine $R,T$, and one of those paths, to construct a double ray $D$ contained in $U_1 \cup U_2$. However, the union of $D$ with $\{x,y\}$ is an \arc{x}{y}\ in \ltp, and this contradicts the choice of $U_1 , U_2$ as arcs are connected spaces.

Thus the subset $\cp' \subseteq \cp$ of paths $P$ that contain a point $q_P$ not in $U_1 \cup U_2$ is infinite. Let $p$ be an accumulation point of $\{q_P \mid P\in \cp'\}$, and note that $p\not\in U_1 \cup U_2$. Easily, $p\in \blg$. Let \seq{p}\ be a sequence of elements of $\{q_P \mid P\in \cp'\}$ converging to $p$. We may assume, \obda, that every $p_i$ is a vertex, for if it is an inner point of the edge $e=uv$, then we can subdivide $e$ into two edges  one of which has length equal to the length of the $u$-$p_i$ half-edge and the other has length equal to the length of the $p_i$-$v$ half-edge. We can now apply \Lr{lcomb} to \seq{p}, and as \g is \lf\ we obtain a comb $K$ that contains \inm\ of the $p_i$ and converges to $p$. On the other hand, the paths in $\cp$ combined with the ray $R$ yield another comb, whose spine is $R$, containing \inm\ of the $p_i$. Combining these two combs it is easy to see that the spine of $K$ is vertex-equivalent to $R$, which means that $p\in \blom$; this however contradicts the fact that $p\not\in U_1 \cup U_2$.

\end{proof}


Having seen \Tr{connbound}, it is natural to wonder whether \blom\ is always path-connected in case \ltp\ is compact. As we shall see, this is not the case.

Gromov \cite{gromov} remarked that for every compact metric space $X$ \ti\ a hyperbolic graph whose hyperbolic boundary is isometric to $X$; thus by \Lr{gromov} every compact space can be obtained as the \ltop-boundary of some \lf\ graph. Our next result strengthens this fact by relaxing the requirement that $X$ be compact.

\begin{theorem} \label{Lind}
Given a metric space $(X,d_X)$, there is a connected \lf\ graph \g and \lER\ \st\ \blg\ is isometric to $X$  \iff\ $X$ is complete and separable.
\end{theorem}
\begin{proof}
For the backward implication, let $U=\{u_1,u_2,\ldots\}$ be a countable dense subset of $X$. To define \G, let its vertex set $V$ consist  of vertices $z^n_u$, one for each \nin\ and $u\in U$. The edge set of \g is constructed as follows. For every  $u\in U$ and every \nin, connect $z^n_u$ to $z^{n+1}_u$ by an edge $e$, and let $\ell(e):= 2^{-n}$. Moreover, \fe\ \nin\ and every pair $u_i,u_j\in U$ \st\ $i,j\leq n$, connect $z^n_{u_i}$ to $z^{n}_{u_j}$ by an edge $e$, and let $\ell(e):= d_X(u_i,u_j)$. 

We now define the required isometry $f: X \to \blg$. For every $u\in U$, let $f(u):= \lim_n z^n_u$. By the choice of $\ell$, it follows easily that $d_\ell(f(u), f(u')) \leq d_X(u,u')$  \fe\ $u,u'\in U$. Applying the triangle inequality for $d_X$ it is also straightforward to check that $d_\ell(f(u), f(u')) \geq d_X(u,u')$. Thus we obtain
\labtequ{isf}{$d_\ell(f(u), f(u'))= d_X(u,u')$  \fe\ $u,u'\in U$}
as desired. For every other point $x\in X$, let \seq{x} be a \Cs\ of points in $U$ converging to $x$. Note that by \eqref{isf} the sequence $(f(x_i))$ is also Cauchy. Thus we may define $f(x):= \lim f(x_i)$. It is an easy consequence of \eqref{isf} and the definition of $f$ that $f$ is distance preserving. 

To see that $f$ is surjective, let $(z^{n_i}_{u_{m_i}})_\iin$ be a \Cs\ (in \ltp) of vertices of \G\ converging to a point $\psi\in \blg$. By the construction of \g and $\ell$, we have $d_X( u_{m_i} , u_{m_j} ) \leq d_\ell( z^{n_i}_{u_{m_i}}, z^{n_j}_{u_{m_j}})$ \fe\ $i,j$, which implies that the sequence $(u_{m_i})_\iin$ of $X$ is also Cauchy. Moreover, since $(z^{n_i}_{u_{m_i}})_\iin$ converges to \blg\ we have $\lim_i n_i= \infty$; thus $(z^{n_i}_{u_{m_i}})_\iin$ is equivalent to the sequence $(f(u_{m_i}))_\iin$. This, and the definition of $f$, implies that $f( \lim u_{m_i} ) = \psi$. We have proved that $f$ is surjective, which means that it is indeed an isometry.

For the forward implication, let \g be a countable graph and fix \lER. Then \blg\ is clearly complete. To show that it is separable, it suffices to show that it is second countable, and a countable basis for \blg\ is $\cu:=\{\blg \cap B_r(x) \mid x\in V(G), r\in \Q\}$ where $B_r(x)$ is the open ball of radius $r$ and center $x$. 

\comment{
	Thus our proof is complete. We remark that  the above construction can be modified slightly to ensure that, in addition, 
\labtequ{ltcom}{if $X$ is compact then \ltp\ is compact as well.}
	Indeed, we can shorten the lengths of the edges of the form $z^n_u z^{n+1}_u$ so that the total length of each ray $z^{1}_{u_i} z^{2}_{u_i} z^{3}_{u_i} \ldots$ is (at most) $2^{-i}$. It \ises\ that \blg\ is still isometric to $X$, and that now every infinite sequence in \ltp\ has an accumulation point as \blg\ is compact, which means that \ltp\ is indeed compact.
}
\end{proof}

If the metric space $X$ in \Tr{Lind} is compact, then it is possible to modify the proof so that \ltp\ is also compact; for example, by using Gromov's aforementioned construction and \Tr{gromov}.
Since there are compact spaces that are connected but not path-connected, this,  and \Tr{connbound}, implies that \blom\ can be non-path-connected for some end \oo\ even if it is compact, although it must be connected.

\note{
\begin{theorem}\label{}
Let \g be a graph and let $\ell_1:E(G)\to \R_{>0}, \ell_2:E(G)\to \R_{>0}$ be two assignments of edge lengths \st\ $\sum_{e\in E(G)} |\ell_1(e)-\ell_2(e)|<\infty$. Then the identity function on $G$ extends to a homeomorphism between $\ltopxl{\ell_1}{G}$ and $\ltopxl{\ell_2}{G}$.
\end{theorem} 
}

\section{Homology and Cycle space} \label{sechom}

\comment{
	\subsection{Introduction}
	As mentioned in the introduction, the \tcs\ outperforms the classical finitary cycle space, i.e.\ the first singular homology group, in that it lets most well known facts about the cycle space of finite graphs generalise to \lf\ graphs. One of the motivations was study \nlf\ graphs, and have a cycle space that behaves well for problems were a notion of edge length is inherent, for which the \tcs\ was doing poorly.

	Thus our aim in this section is to develop a definition, for every graph $G$ and every $\lER$, a homology group $\cclg$ that generalises the \tcs\ of Diestel and K\"uhn and takes edge lengths into account, in particular, it behaves well with respect to the example in \fig{fan}... and makes everything true. 
}

\subsection{Introduction} \label{sechomi}

In this section we are going to describe a new homology $\lhomx{X}$ defined on an arbitrary metric space $X$. It is  proved in \cite{lhom} that $\lhomx{X}$ coincides, in dimension 1, with the \tcs\ \ccg\ discussed in \Sr{insuccc} if applied to the metric space \ltp\ for the right function $\ell$, and that  $\lhomx{X}$ extends properties of \ccg\ to every compact metric space $X$. We will discuss, in particular, how \lhom\ can be used to extend results about the topological cycle space of a \lfg\ to \nlf\ graphs.

In order to define \ccg\ for a \lf\ graph $G$, we call a set of edges $D\subseteq E(G)$ a \defi{circuit}, if there is a topological circle $C$ in \fcg\ that traverses every edge in $D$ but no edge outside $D$. Then, let \ccg\ be the vector space over $\Z_2$ consisting of those subsets $F$ of $E(G)$ such that $F$ can be written as the sum of a possibly infinite family of \defi{circuits}; such a sum is well-defined, and allowed, if and only if the family is \defi{thin}, that is, no edge appears in infinitely many summands: the sum of such a family $\cf$ is by definition the set of edges that appear in an odd number of members of $\cf$.

The study of \ccg\ has been a very active field lately \cite{cyclesIntro}, but \nlf\ graphs have received little attention. One reason for this is the fact that \fcg\ can be generalised to \nlf\ graphs in several ways (see \Sr{secNlf}), all having advantages and disadvantages, and it was not clear which of them is the right topology on which the definition of \ccc\ should be based. Let me illustrate this by an example. In the graph of \fig{fan} (ignore the indicated edge lengths for the time being), let $\cf$ be the family of all triangles incident with $x$. Note that $\cf$ is a {thin} family, and its sum comprises the edge $xz$ and all edges of the horizontal ray. Now such an edge set can only be considered to be a circuit if the topology chosen identifies $x$ with the end \oo\ of the graph. Similarly, the vertex $y$ also has to be identified with \oo, and thus also with $x$, if such sums are to be allowed. So we either have to forbid these infinite sums, or content ourselves with a topology that changes the structure of the graph by identifying vertices. Both approaches have been considered \cite{cyclesII,tst}, but none was pursued very far. 

\showFig{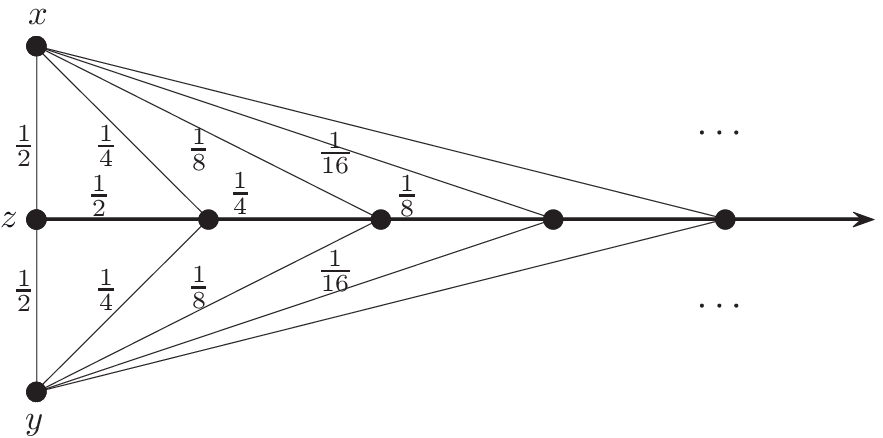}{A simple \nlfg.}

Let me now argue that with \ltp\ we might be able to overcome this dilemma: our aim is to define, for a graph $G$ and \lER, a cycle space \ccl\ based on circles in the topology \ltp, and prove results that hold for every $\ell$ and generalise the properties of the cycle space of a finite graph. As we shall see, this would allow us to postpone the decision of whether to identify vertices or forbid some infinite sums until we have a certain application in mind. See also the remark at the end of \Sr{secNlf}.

Suppose for example that we allow in \ccl\ precisely those sums of families \fml{F}\ of \defi{circuits}, that is, edge sets of circles in \ltp, that have finite total length, i.e.\ satisfy $\sum_{i} \sum_{e\in F_i} \ell(e)<\infty$. Firstly, note that such a family is always thin, since if an edge $e$ lies in infinitely many $F_i$ then the total length of \fml{F}\ will be infinite. We will now consider two special assignments $\ell_1,\ell_2$ of lengths to the edges of the graph \g of \fig{fan}, and see how \ccl\ behaves in these two interesting cases. 

Let $\ell_1$ be an assignment of lengths as shown in \fig{fan}. Then, the family of all triangles of \g has finite total length, so their sum lies in \ccl. This sum is the edge set $\{xz,zy\}$. On the other hand, $x$ and $y$ have distance $d_{\ell_1}(x,y)=0$ in this case, which means that $x$ and $y$ are identified in \ltp\ and $\{xz,zy\}$ is indeed a circuit.

For the second case, let $\ell_2(e)=1$ for every $e\in E(G)$. Then, $x$ and $y$ are not identified, so $\{xz,zy\}$ is not a circuit, but it is also not in \ccl, since every infinite family of circuits has infinite total length. 

The interested reader will try other assignments of edge lengths for this graph and convince himself that \ccl\ always behaves well in the sense that for any \lER\ any element of \ccl\ can be written as the sum of a family of pairwise  edge disjoint circuits.

Although \ccl\ performs well in simple cases like \fig{fan}, it is a rather naive concept, since in general, even if \g is \lf, there can be arcs and circles in \ltp\ that contain no vertex and no edge of \G; this can be the case for example when \ltp\ is the hyperbolic compactification of \G, see \Sr{secHyp}. Thus circuits do not describe their circles accurately enough. To make matters worse, even if we decide to disregard those circles that have subarcs contained in the boundary, and consider only those circles $C$ for which $E(C)$ is dense in $C$ we will not obtain a satisfactory cycle space as we shall see in the next section.

For these reasons, we will follow an approach combining singular homology with the above ideas. This will lead us, in \Sr{lhom}, to the homology \lhom\ that apart for graphs endowed with \ltop\ can be defined for arbitrary metric spaces.

\subsection{Failure of the edge-set approach} \label{faile}

Given a graph \g and \lER, let us call a (topological) circle $C$ in \ltp\ \defi{proper} if $E(C)$ is dense in $C$. Following up the above discussion, let us define \ccl\ to be the vector space over $\Z_2$ consisting of those subsets $F$ of $E(G)$ such that $F$ can be written as the sum of a (thin) family $\fml{F}$ of {circuits} of proper topological circles in \ltp\ such that $\sum_i \sum_{e\in F_i} \ell(e) < \infty$. Although \ccl\ behaves well with respect to the graph in \fig{fan}, it turns out that it is not a good concept: we will now construct an example in which \ccl\ has an element that contains no circuit. The reader who is already convinced that circuits cannot describe the circles in \ltp\ accurately enough could skip to \Sr{lhom}, however it is advisable to go through the following example anyway, as it will facilitate the understanding of the homology \lhom\ introduced there.

\example{fomalhaut}{
\normalfont

The graph \g we are going to construct contains a \procir\ $C_0$ like in \Er{ant} (\fig{antares}), the length of which is larger than the sum of the lengths of its edges. Moreover, it is possible to replace each edge of $C_0$ by an arc which also has a length larger than the sum of the lengths of its edges to obtain a new \procir\ $C_1$. We then replace each edge of $C_1$ by such an arc to obtain the \procir\ $C_2$, and so on. The \procir s $C_i$ will have lengths that grow slightly\afsubm{bla}\ with $i$, but the proportion of that length accounted for by the edges will tend to 0 as $i$ grows to infinity. Moreover, the $C_i$ will converge\afsubm{``''}\ to a circle $C$ containing no edges at all. This is the point where $\ccl(G)$ will break down: we will define a family $(F_i)$ of sums of circuits such that for every $j\in\N$ the edge set $\sum_{i\leq j} F_j$ is the circuit of $C_j$, but $\sum_{i\in\N} F_j=0$. Thus $(F_i)$ tends to describe the circles $C_i$, but its ``limit'' fails to describe their limit $C$, and we will exploit this fact to construct pathological elements of $\ccl(G)$.

We will perform the construction of both the underlying graph \g and the required family recursively (and simultaneously), in \oo\ steps. After each step $i$, we will have defined the graph $G_i$ and the edge-sets $F_0, \ldots F_{i}$, with $F_j\subseteq E(G_j)$, so that $\sum_{0\leq j \le i} F_j$ is the circuit of a \procir\ $C_i$ such that\afsubm{lengths grew}\ 
$\ell(E(C_i))\leq 2^{-i}$. The graph in which the pathological elements of \ccl(G)\ live is $G:=\bigcup G_i$. In each step we will specify the lengths of the newly added edges, and we will choose them in such a way that for every $i$ and every two vertices $x,y\in G_i$ the distance  $d_\ell(x,y)$ is not influenced by edges added after step $i$; in other words, no \pth{x}{y}\ in $G$ has a length less than the distance between $x$ and $y$ in $G_i$. This implies that any circle of $G_i$ is also a circle of $G$.

Formally, for step $0$, let $G_0$ be the graph in \Er{ant} (with edge lengths as in that example), let $C_0$ be the thick, ``wild'' circle there, and let $F_0=E(C_0)$ be its circuit. Then, for $i= 1,2,\ldots$, suppose we have already performed step $i$ so that it satisfies the above requirements. In step $i+1$, for every edge $e=vw$ of $C_i$, take a copy $H_e$ of the graph of \Er{ant}, and let $W_e$ denote the circle of $H_e$ corresponding to the thick circle $W$ in \Er{ant}. Recall that $\ell(W)=2\half$ but $\sum_{e \in E(W)} \ell(e)=1\half$. The vertex of $H_e$ corresponding to the vertex $x$ in \fig{antares} divides the outer double ray $L_e$ into two rays $R,S$;  now join, for every $j\in \N$, the $j$th vertex of $R$ to the endpoint $v$ of $e$ by an edge of length $\ell(e)/2^{j-1}$ and also join the $j$th vertex of $S$ to the other endpoint $w$ of $e$ by an edge of length $\ell(e)/2^{j-1}$; see \fig{glueingNLF}. Note that $v$ and $w$ have infinite degree now. Later we will see how the construction can be modified to obtain a \lfg\ with similar properties. 

\showFig{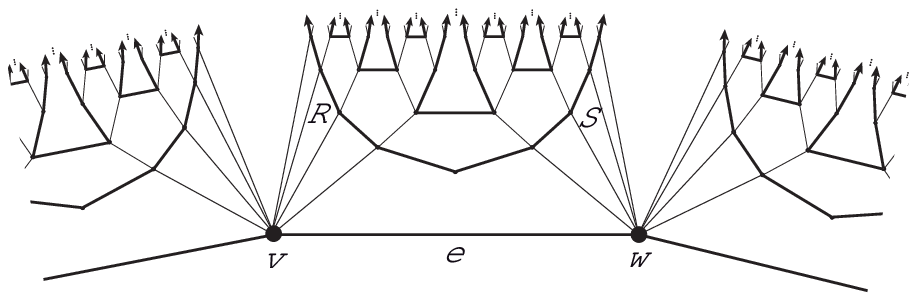}{Joining the graphs $H_e$ to $G_i$ to obtain $G_{i+1}$.}

Now change the lengths of the edges of $H_e$ as follows. Scale the lengths of 
the edges in $E(H_e)$ down 
in such a way that 
the distance of the endpoints of $W_e \sm L_e$ is $\ell(e)$ and 
so that
\labtequ{half}{$\sum_{f\in E(W_e)\sm E(L_e)} \ell(f)= \ell(e)/2$.}
Such a choice of edge-lengths is possible since in \Er{ant} we were allowed to choose $s$ and $c$ independently from each other. \afsubm{bla}

Note that by the choice of the lengths of the newly attached edges, no \pth{v}{w}\ going through $H_e$ has a length less than $\ell(e)$. Moreover, since the lengths of the newly added edges converge to 0, the distance between $v$ and the end of the ray $R$ is 0, and the distance between $w$ and the end of the ray $S$ is also 0. Thus the union of $e$ with the arc $W_e\sm L_e$ is a \tocir\ $C_e$.

Having performed this operation for every edge $e$ of $C_i$ we obtain the new graph $G_{i+1}$. We now let $F_{i+1}:=\sum_{e\in E(C_i)} E(C_e)$. Moreover, let $C_{i+1}$ be the \tocir\ in $\ltopx{G_{i+1}}$ obtained from $C_i$ by replacing each edge $e\in E(C_i)$ by the arc $C_e \sm \kreis{e}=W_e\sm L_e$. This completes step $i+1$. Note that 
\labtequ{fi}{$E(C_{i+1})= E(C_{i}) + F_{i+1}$,}
thus assuming, inductively, that $\sum_{j\leq i} F_j= E(C_{i}) $, we obtain $\sum_{j\leq i+1} F_j= E(C_{i+1})$. 

Let $\ell_i:=\sum_{e\in F_i} \ell(e)$. As the circles $C_i$ and $C_{i+1}$ are edge-disjoint, \eqref{fi} implies $F_i= E(C_i) \cup E(C_{i+1})$, and thus $$\ell_i = \sum_{e\in E(C_{i+1})} \ell(e) + \sum_{e\in E(C_{i})} \ell(e).$$ By~\eqref{half} and the definition of $C_{i+1}$ we obtain \afsubm{=} $\sum_{e\in E(C_{i+1})} \ell(e) = \half \sum_{e\in E(C_{i})} \ell(e)$. Plugging this inequality  into the above equation twice (for two subsequent values of $i$), yields \afsubm{=} $\ell_i =  \frac{\ell_{i-1}}{2}$. Thus the family $\seq{F}$ does have finite total length. 

Now as $F_i= E(C_i) \cup E(C_{i+1})$ and the circles $C_i$ and $C_{i+1}$ are edge-disjoint, every edge in $\bigcup F_i$ appears in precisely two members of $(F_i)$ and thus $\sum_{i\in\N} F_i=\emptyset$. We can now slightly modify the family $(F_i)$ to obtain a new finite-length family $(F'_i)$ of circuits of \procir s in $G=\bigcup G_i$ the sum of which is a single edge: pick an edge $f$ of $C_0$, remove  from $F_1$ the circuit of $C_f$, then remove from $F_2$ all circuits of circles $C_e$ corresponding to edges $e$ that lie in $C_f$, and so on. For the resulting family $(F'_i)$ we then have $\sum_{i\in\N} F'_i=\{f\}$, and as $(F'_i)$ has finite total length the singleton $\{f\}$ is an element of \ccl, and in fact a pathological one as the endvertices of $f$ are distinct points in \ltp.

Transforming this example to a \lf\ one with the same properties is easy. The intuitive idea is to replace each vertex of infinite degree of $G$ by an end and its incident edges by double rays ``connecting" the corresponding ends. 
More precisely, let $H$ be a copy of the graph of \fig{antares}, and subdivide every edge $e=uv$ of $H$ into two edges $ux_e,x_ev$ to obtain the graph $H'$; assign lengths to the new edges so that $\ell(ux_e)+\ell(x_ev)=\ell(uv)$. Then, for every vertex $u\in V(H')\cap (V(H) \sm V(L))$, replace each (subdivided) edge $ux_e$ incident with $u$ by a ray starting at $x_e$ and having total length $\ell(ux_e)$, and make all these rays converge to the same point by joining them by infinitely many new edges with lengths tending to zero. Denote the end containing these (now equivalent) rays by $\oo_u$, and note that the lengths of the newly added edges can be chosen so that the distance between any two such ends $\oo_u, \oo_v$ of the new graph equals the distance between $u$ and $v$ in $H$. This gives rise to a new graph $H''$, and we may use $H''$ instead of $H$ as a building block in the above construction of \G, to obtain a \lf\ graph similar to \G; instead of the attaching operation of \fig{glueingNLF} we would now use an operation as indicated in \fig{glueing} (and instead of attaching, at step $i+1$, a new copy of $H''$ for each edge of $C_i$ we will attach a new copy of $H''$ for each double ray in $C_i$).

\showFig{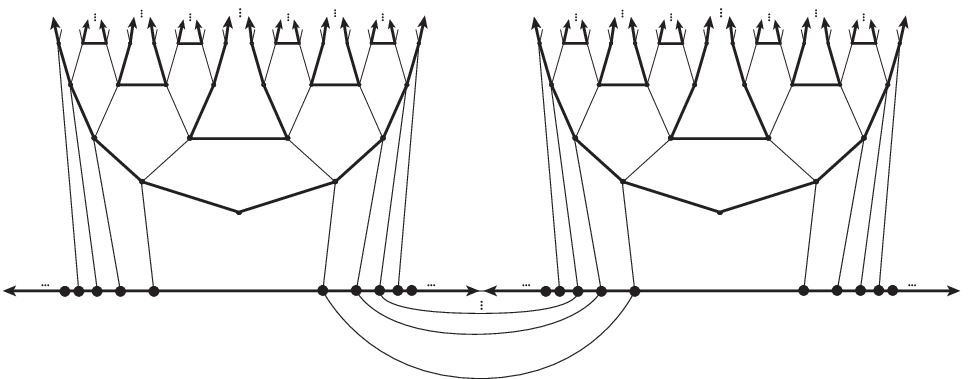}{Modifying the construction of \g to make it \lf.}

} 

Example~\ref{fomalhaut} and our previous discussion show that the edge-sets of circles do not reflect the structure of a graph accurately enough, so we will take another tack: in the following section we are going to study circles from the point of view of singular homology.

\subsection{The singular-homology point of view} \label{lhom}

The \tcs\ \ccc\ of a \lfg\ \G, discussed in Sections~\ref{insuccc} and \ref{sechomi}, bears some similarity to the first singular homology group $H_1$ of \fcg, as both are based on circles, i.e.\ closed 1-simplices, in \fcg. \RD\ and Spr\"ussel  \cite{Hom1} investigated the precise relationship between \ccg\ and $H_1(\fcg)$, and found out that they are not the same: they defined a canonical mapping $f: H_1(\fcg) \to \ccg$ that assigns to every class $c\in H_1(\fcg)$ the set of edges traversed by one (and thus, as they prove, by every) representative of $c$ an odd number of times ---assuming both \ccc\ and $H_1$ are defined over the group $\Z_2$--- and showed that $f$ is a group homomorphism that is surjective, but not necessarily injective.  \Er{Ekringel} below shows their construction of a non-zero element of $H_1(\fcg)$ that $f$ maps to the zero element of $\ccg$.

\example{Ekringel}{
\normalfont

The space $X$ (solid lines) depicted in \fig{kringel} can be thought of as either the \FC\ \fcg\ of the infinite ladder \G, or as a subset of the euclidean plane; note that the two spaces are homeomorphic, since the set of vertices converges to a single limit point in both of them. 

Let \sig\ be the 1-simplex indicated by the dashed curve. This simplex starts and ends at the upper-left vertex, traverses every horizontal edge twice in each direction, and  traverses every vertical edge once in each direction. Thus it is mapped to the zero element of \ccg\ by the mapping $f$. However, \RD\ and Spr\"ussel  \cite{Hom1} proved that \sig\ does not belong to the zero element of $H_1(\fcg)$.

\showFig{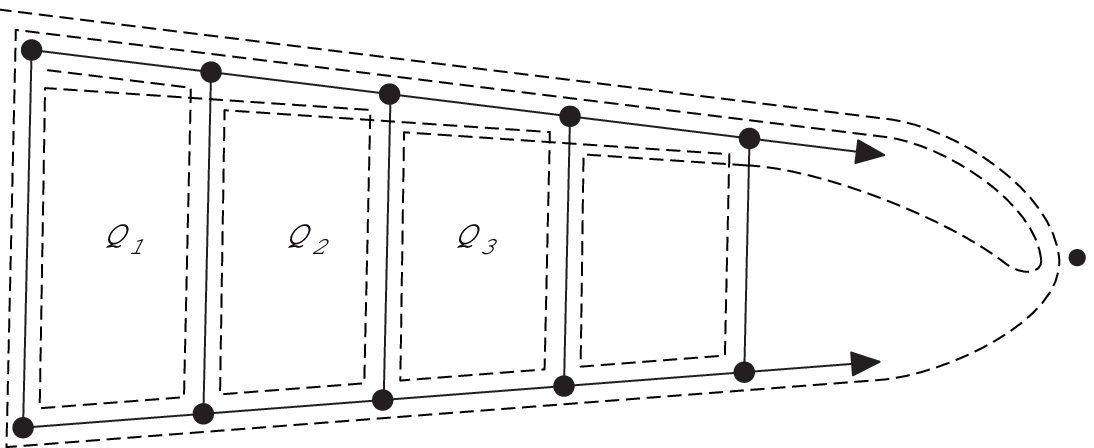}{The 1-simplex of \RD\ and Spr\"ussel. It traverses each edge the same number of times in each direction, but it is not null-homologous.}
}

In what follows we are going to modify $H_1$ into a new homology group \lhom\ that does coincide with \ccc\ if applied to \fcg, and moreover generalises properties of \ccc\ when applied to an arbitrary continuum. The main idea is to impose a pseudo-metric on $H_1$ and identify elements with each other if their distance is $0$. (Here we will constrain ourselves to dimension 1 for simplicity, but the construction can be carried out for any dimension, see \cite{lhom}.)

For this, let $X$ be a metric space and define an \defi{area-extension} of $X$ to be a metric space  $X'$, in which $X$ is embedded by a fixed isometry $i: X \to X'$, such that each component of $X' \sm i(X)$ is either a disc or a cylinder. The \defi{\evol} of this extension is the sum of the \vol\s\ of the components of $X' \sm i(X)$. 

We now define a pseudo-metric $d_1$ on the first singular homology group $H_1(X)$ of $X$. Given two elements $[\phi],[\chi]$ of $H_1(X)$, where $\phi$ and $\chi$ are $1$-chains, let $d_1([\phi],[\chi])$ be the infimum of the \evol\s\ of all area-extensions $X'$ of $X$ \st\ $\phi$ and $\chi$ belong to the same element of $H_1(X')$.

It follows easily by the definitions that $d_1$ satisfies the triangle inequality. However, $d_1$ is not a metric, since there may exist $c, f\in H_1$, $c\neq f$, with $d_1(c,f)=0$; indeed, for the class $c$ of the simplex of \Er{Ekringel} we have $d_1(c,\zero)=0$, although $c$ is not the zero element $\zero$ of $H_1$; to see that $d_1(c,\zero)=0$, let $Q_j$ be the $j$th 4-gon in the space $X$ of \Er{Ekringel}, and consider the area-extension $X^i$ of $X$  obtained by pasting to $X$ the plane discs bounded by all 4-gons $Q_j$ for $j>i$. Easily, $X$ is indeed isometrically embedded in $X^i$. It is straightforward to check that the simplex $\sig$ of \Er{Ekringel} is null-homologous in $X^i$. Moreover, the area of $X^i$ if finite \fe\ $i$, and this area tends to 0 as $i$ grows to infinity. Thus \btdo\ $d_1$, we have $d_1(c,\zero)=0$ as claimed.

Declaring $c,f\in H_1$ to be equivalent if $d_1(c,f)=0$ and taking the quotient with respect to this equivalence relation we obtain the group $\lhomp=\lhomp(X)$; the group operation on $\lhomp$ can be naturally defined for every $c,d\in \lhomp$ by letting $c+d$ be the equivalence class of $\alp+\beta$ where $\alp\in c$ and  $\beta\in d$. To see that this sum is well defined, i.e.\ does not depend on the choice of $\alp$ and $\beta$, note that the union of two area-extensions of $X$, of \evol\ at most \eps\ each, is an area-extension of $X$ of \evol\ at most $2\eps$.

Now $d_1$ induces a distance function on $\lhomp$, which we will, with a slight abuse, still denote by $d_1$. It is not hard to prove that $d_1$ is now a metric on $\lhomp$; see \cite{lhom} for details. We now define our desired group $\lhom=\lhomx{X}$ to be the completion of the metric space $(H'_1,d)$. (The operation of $\lhom$ is defined, \fe\ $C,D\in \lhom$, by $C+D:= \lim_i (c_i + d_i)$ where \seq{c}\ is a \Cs\ in $C$ and \seq{d}\ is a \Cs\ in $D$.)

It can be proved that \ccg, where \g is a \lfg, is a special case of \lhom:

\begin{theorem}[\cite{lhom}] \label{lhiscc}
If \finl\ then $\lhom(\ltp)$ is isomorphic to \ccg.
\end{theorem} 
(Recall that if \finl\ then $\ltp \eqT \fcg$ by \Tr{finlf}.)
The isomorphism of \Tr{lhiscc} is canonical, assigning to the class corresponding to a circle $C$ in \ltp\ the set of edges traversed by $C$.

The interested reader will check that \lhom\ also behaves well \wrt\ \Er{fomalhaut}. The important observation is that \fe\ $i$ there is an area-extension $X^i$ of the space \ltp\  in which all circles $C_j$ for $j>i$ become homologous to each other, and by the choice of the lengths of the circles $W_e$, the $X^i$ can be constructed so that the area of $X^i$ tends to 0 as $i$ grows to infinity.

Having seen \Tr{lhiscc}, we could try to use \lhom\ to extend results about the cycle space of a finite or \lfg\ to other spaces, for example \nlfg\s\ (recall our discussion in \Sr{sechomi}). But in order to be able to do so we would first have to interpret those results into a more ``singular'' language. In this and the 
 next section 
we are going to see three examples of how this can be accomplished.

The main result of \cite{lhom} is that \lhom\ extends the following fundamental property of \ccg:

\LemCycDec{cycD}

\Lr{cycD} has found several applications in the study of \ccg\ \cite{LocFinMacLane,cyclesI,geo} and elsewhere \cite{fleisch}, and several proofs have been published \cite{hp}. The following theorem generalises \Lr{cycD} to $\lhom(X)$ in case $X$ is a compact metric space. 

Recall that an 1-cycle is a (finite) formal sum of 1-simplices that lies in the kernel of the boundary operator $\partial_1: C_1 \to C_0$ (\cite{Hatcher}). A \defi{\rep} of $C\in \lhom$ is an infinite sequence $\seq{z}$ of $1$-cycles $z_i$  such that the sequence $(\ec{\sum_{j\leq i} z_j})_{i\in\N}$ is a \Cs\ in $C$, where \defi{$\ec{z}$} denotes the element of \lhomp\ corresponding to the $1$-cycle $z$. (One can think of a \rep\ as an ``1-cycle'' comprising infinitely many 1-simplices.) For an 1-cycle $z$ we define its \defi{length $\ell(z)$} to be the sum of the lengths of the simplices appearing in $z$ with a non-zero coefficient. Note that the coefficients of the 1-simplices in $z$ do not play any role in the definition of $\ell(z)$ as long as they are non-zero; in particular, $\ell(z) \geq 0$ for every $z$.

\begin{theorem}[\cite{lhom}]\label{lCycD}
For every compact metric space $X$ and $C \in \lhomx{X}$, there is a \rep\ $\seq{z}$ of $C$ that minimizes $\sum_i \ell(z_i)$ among all \rep\s\ of $C$. Moreover, \seq{z} can be chosen so that each $z_i$ is a \clex. 
\end{theorem} 

(See \Sr{defsTop} for the definition of \clex.) 

\Tr{lCycD} in conjunction with \Tr{lhiscc} imply in particular \Lr{cycD}. Indeed, given a \lfg\ \G, pick an assignment \lER\ with \finl; now each element $C$ of \ccg\ corresponds, by \Tr{lhiscc}, to an element $f(C)$ of $\lhomx{\ltp}=\lhomx{\fcg}$ (where we applied \Tr{finlf}). Applying \Tr{lCycD} to $f(C)$ we obtain the representative $\seq{z}$ where each $z_i$ is a \clex. Note that no edge $e$ can be traversed by both $z_i$ and $z_j$ for $i\neq j$, for if this was a case then we could ``glue'' $z_i$ and $z_j$ together after removing the edge $e$ from both, to obtain a new \rep\ $\seq{z'}$ of $f(C)$ with $\sum_i \ell(z'_i) = \sum_i \ell(z_i) - 2\ell(e)$, contradicting the choice of $(z_i)$. It follows that $\{E(z_i) \mid \iin \}$ is a set of pairwise disjoint circuits whose union is $C$.

\section{Geodetic circles and MacLane's planarity criterion} \label{secgeo}



A cycle in a graph \g is \defi{geodetic}, if for every $x,y\in V(C)$ one of the two $\pths{x}{y}$ in $C$ is a shortest $\pth{x}{y}$ also in the whole graph \G.
More generally, a circle $C$ in a metric space $(X,d)$ is \defi{geodetic}, if for every $x,y\in C$ one of the two $\arcs{x}{y}$ in $C$ has length $d(x,y)$. If $X$ is a finite graph then it is an easy, but interesting, fact that its geodetic circles generate its cycle space \cite{geo}. The following conjecture is an attempt to generalise this fact to an arbitrary continuum, or at least to an arbitrary compact \ltp, using the homology group \lhom\ of \Sr{sechom}.

For a set $U\subseteq \lhom$, or $U\subseteq \lhomp$, let $\wsp{U}$ be the set of elements of $\lhom$ that can be written as a sum of finitely many elements of $U$, and define the \defi{span} $\asp{U}$ of $U$ to be the closure  of \wsp{U} in (the metric space) \lhom. We say that $U$ \defi{spans} \lhom\ if $\asp{U}=\lhom$.

Call a \clex\ \defi{geodetic} if its image is a geodetic circle.

\begin{conjecture}\label{conjGeohH}
Let $G,\ell$ be such that \ltp\ is compact, and let $U$ be the set of elements $\ec{1\chi} \in \lhomp$ \st\ $\chi$ is a geodetic \clex. Then $U$ {spans} \lhom.
\end{conjecture}
	
(Again, $\ec{z}$ denotes the element of \lhomp\ corresponding to the $1$-cycle $z$.) 

\Cnr{conjGeohH} could also be formulated for an arbitrary compact metric space instead of \ltp.

In \cite{geo} a variant of 	\Cnr{conjGeohH} was proved for the special case when $\ltp \eqT \fcg$, although, the notion of geodetic circle used there was slightly different: the length of an arc or circle was taken to be the sum of the lengths of the edges it contains, and geodetic circles were defined with respect to that notion of length. However, in the special case when \finl, \Lr{finl} and \Lr{lcisle} imply that our notion of geodetic circle coincides with that of \cite{geo}, and thus the main result of that paper implies that \Cnr{conjGeohH} is true if \finl. Conversely, a proof of \Cnr{conjGeohH} would imply the main result of \cite{geo} for that case. 

If we drop the requirement that \ltp\ be compact in \Cnr{conjGeohH} then it becomes false as shown by the following example.

\example{exnogeo}{
\normalfont

In the graph of \fig{nogeo} no geodetic circle  contains the edge $e$. To prove this, let us first claim that any geodetic circle $C$ containing $e$ must visit both \bp s $x,y$; for if not, then $C$ must contain a vertical edge $f=uv$ from the upper or the bottom row, and then it is possible to shortcut $C$ by replacing $f$ by a finite \pth{u}{v} to the right of $f$ that is shorter than $f$. Thus, $C$ must contain an \arc{x}{y}, but it is easy to see that there is no shortest \arc{x}{y}, so $C$ cannot be geodetic. It is now straightforward to prove that if $D$ is an element of \lhom\ corresponding to a cycle of this graph containing $e$, then $D$ is not in the span \asp{U}\ of the set $U$ of \Cnr{conjGeohH}.
}

\showFig{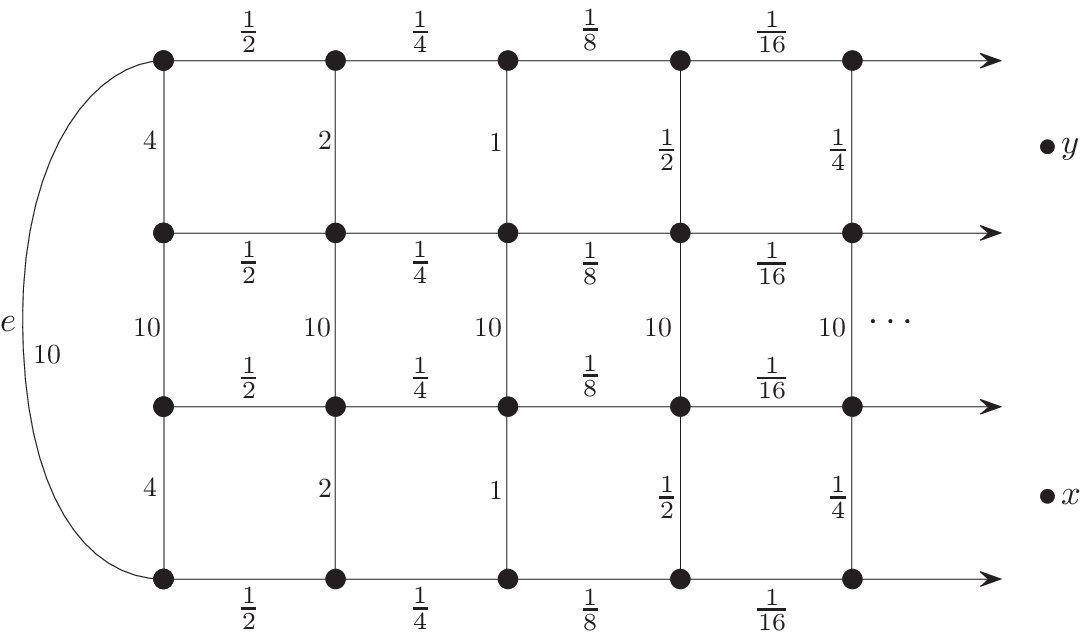}{An example showing that \Cnr{conjGeohH} is false if \ltp\ is not compact.}

MacLane's well-known planarity criterion chartacterises planar graphs in algebraic terms (see also \cite[Theorem 4.5.1]{diestelBook05}):

\ThmMacLane{finml}

A generating set of \ccg\ is called \defi{simple} if no edge appears in more than two of its elements. 

\BS\ generalised \Tr{finml} to \lfg\s\ \cite{LocFinMacLane}.
Our next conjecture is an attempt to characterise all planar continua in algebraic terms in a way similar to
\Tr{finml}. In order to state it, call a set $S$ of \clexs\ in a metric space $X$ \defi{simple} if for every 1-simplex \sig\ in $X$ there are at most 2 elements of $S$ that have a  sub-simplex homotopic to \sig.

\begin{conjecture}\label{conjMcL}
Let $X$ be a compact, 1--dimensional, locally connected, metrizable space that has no cut point. Then $X$ is embeddable in $S^2$ \iff\ \ti\ a simple set $S$ of \clexs\ in $X$ and  a metric $d$ inducing the topology of $X$ so that the set $U:= \{ \ec{\chi} \in \lhomx{X} \mid \chi \in S\}$ {spans} \lhom.
\end{conjecture}
	
To see the necessity of assuming local connectedness in this conjecture, consider the space obtained from some non-planar graph after replacing an interval of each edge by a \defi{topologist's sine curve}. 

The requirement that $X$ have no cut point in the conjecture is also necessary. Indeed, let $G$ be the infinite 2-dimensional grid, let $I$ be the real unit interval, and let $X$ be the topological space obtained from \fcg\ and $I$ by identifying the end of $G$ with an endpoint of $I$ ---which is now a cut point of $X$. It is easy to see that $X$ is not embeddable in $S^2$, although it satisfies all other requirements of \Cnr{conjMcL}.
	
\comment{The following is not strong enough to prove \Tr{GS}
	\begin{conjecture}\label{conjGeohH}
	For every $c \in \hH$, there is a \Cs\ $\seq{\chi} \in c$ every element $\chi_i$ of which is the homology class of a (finite) sum of geodetic \clex es.
	\end{conjecture}
}

\note{
	\begin{conjecture}\label{conjGeoS}
If $G$ is \lf\ then $S(\geol)=\cclg$.
\end{conjecture} 

\begin{conjecture}\label{conjGeo}
If $G$ is \lf\ then $<\geol>=\cclg$.
\end{conjecture} 

If $\ltp \eqT \fcg$ then $(|G|,\ell)$ is a metric representation of $G$: \ltp\ balls are smaller and $\fcg$ basic open sets clearly contain "old" balls. However, the concept of geodetic in \cite{geo} is different, since only the length along edges is taken into account.

	Lemma 4.10 of Geo holds with the new definition of geodesic circle: If $(C_i)$ is a chain of geodesic cycles, then from some $i$ on all $C_i$ have the same length, and you can construct the limit bearing in mind this length.
}

\section{Line graphs} \label{secLg}

In this section we are going to prove \Tr{Lg}; let us repeat it here:
\begin{noTheorem}
If $G$ is a countable graph and $\eti$ has a \tet\ then $\mtopx{L(G)}$ has a \hcy.
\end{noTheorem}
Before proving this, let me argue that it is indeed necessary to use both topologies $\etop$ and $\mtop$ in its statement (we can replace $\etop$ by $\GetopG$ though). 

Firstly, if \g is a \nlfg\ then \mtp\ never has a \tet: for if $\sig: S^1 \to \mtp$ is such a \tet\ and $x$ a vertex of \g of infinite degree, then as \sig\ has to traverse all the edges incident with $x$, its domain $S^1$ must contain a point $z$ that is an accumulation point of an infinite set of intervals each of which is mapped to a distinct edge of $x$. This point $z$ can only be mapped to $x$ by \sig, but as $x$ has open neighbourhoods in \mtp\ that contain no other vertex except $x$, \sig\ fails to be continuous at the point $z$, a contradiction. On the other hand, if \g is a countable connected graph then it follows from  Lemmas~\ref{etoc} and~\ref{etoconv} below that the existence of a \tet\ in \eti\ is equivalent to the assertion that \g has no finite cut of odd cardinality. Thus the ``right'' topology to look for a \tet\ is \etop.

What if we try to replace $\mtopx{L(G)}$ in \Tr{Lg} by $\etop(L(G))$? Consider the graph $G$ in \fig{figLg}. For this graph we can easily find a \tet\ in \eti. 
Moreover, $\etop(L(G))$ consists of $L(G)$ and precisely one additional point $\oo$, namely, the equivalence class containing $\{e,f,g,h\}$ and the unique edge-end of $L(G)$. Thus, if there is  a \hcy\ $H$ in $\etop(L(G))$, then $H$ must consist of $\oo$ and a double ray containing all vertices of $L(G)$ except $e,f,g$ and $h$. But it is easy to see that such a ray does not exist in $L(G)$. One way to go around this is to consider the topology $\etop'(L(G))$ istead of $\etop(L(G))$, that is, consider $\{e,f,g,h\}$ as distinct points with the same set of open neighbourhoods. In this case we would obtain a \hcy\ in $\etop'(L(G))$, but the interested reader can check that no such \hcy\ $H$ can have the property that every vertex of $L(G)$ is incident with two edges in $H$: some of the vertices $e,f,g,h$ would have to be an accumulation point in $H$ of other vertices. The \hcy\ $H$ we are going to construct in the proof of \Tr{Lg} however, does have the property that every vertex $x$ is incident with two edges in $H$: since $x$ has open neighbourhoods in $\mtopx{L(G)}$ that contain no other vertex, a continuous injective mapping $\tau: S^1 \to \mtopx{L(G)}$ can only reach $x$ along an edge and must use another edge to leave it.

\showFig{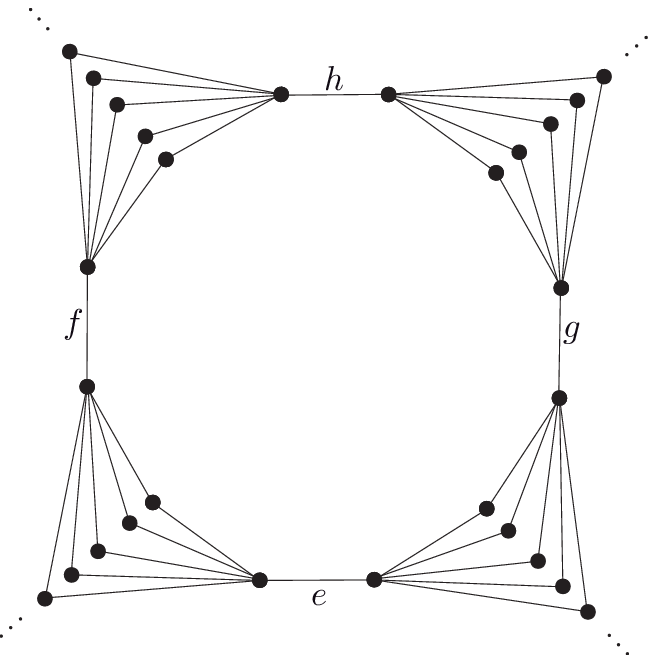}{A case in which $\etop(G)$ is eulerian but $\etop(L(G))$ is not hamiltonian.}

We will now prove \Tr{Lg}. In order to do so we will first need to characterize the graphs \g for which \eti\ has a \tet. This is achieved by the following two lemmas.

\begin{lemma}\label{etoc}
If \eti\ or \Getop\ has a \tet\ then \g has no finite cut of odd cardinality.
\end{lemma}
\begin{proof}
Since a \tet\ in \Getop\ easily induces a \tet\ in \eti\ (by replacing each point with its equivalence class) it suffices to prove the assertion for \eti.
Let $F$ be a finite cut of $G$, and let the set $F'\subset \eti$ consist of a choice of one inner point from each edge in $F$. Then by the definition of \Getop, $\eti\ \sm F'$ is the union of two disjoint open sets. It is now easy to see that any continuous image of $S^1$ in \eti\ must ``cross'' $F$ an even number of times, proving that if $|F|$ is odd then \eti\ cannot have a \tet.
\end{proof}

In \cite[Theorem~4]{fleisch} it was proved that if for a \lfg\ \g there is a \tet\ \sig\ in \fcg, then \sig\ can be chosen so that it visits every end of \g precisely once. Our next lemma extends this to a countable \G. The aforementioned statement was used in the same paper to prove the \lf\ version of \Tr{Lg}, by starting with such a \tet\ \sig\ of \g and modifying it in the obvious way to obtain a \hcy\ $H$ of $L(G)$; the extra condition on \sig\ was used there to make sure that $H$ visits no end of $L(G)$ more than once. 
Similarly, the now more elaborate condition on \sig\ in the following lemma will be necessary in order to make sure that the \hcy\ we will be constructing in the proof of \Tr{Lg} will be injective at the boundary points.

\begin{lemma}\label{etoconv}
If \g is a countable connected graph that has no finite cut of odd cardinality then \Getop\ has a \tet\ \sig. Moreover, \sig\ can be chosen so that for every  two distinct points $x,y\in S^1$ each of which is an accumulation point of preimages (under \sig) of vertices of \G, the points $\sig(x)$ and $\sig(y)$ can be separated in \Getop\ by finitely many edges (in particular, $\sig(x)\neq \sig(y)$).
\end{lemma}
\begin{proof}
Clearly $G$ has a finite cycle $C$, because otherwise every edge would be a cut of cardinality 1. Fix a continuous function  $\sigma_0 : S^1 \to C$, that maps a closed, non-trivial interval of $S^1$ to each vertex and edge of $C$ (here an edge contains its endvertices). 

We will now inductively, in $\omega$ steps, define a \tet\ $\sigma$ in \Getop. After each step $i$ we will have defined a finite set of edges $F_i$ and a continuous surjection $\sigma_i : S^1 \to \overline{F_i}$, where $\overline{F_i}$ is the subspace of $\Getop$ consisting of all edges in $F_i$ and their incident vertices. In addition, we will have chosen a set of vertices $S_i$ in $\overline{F_i}$, with the property that every component of $G - F_i$ is incident with at most one vertex in $S_i$. For each vertex $v \in S_i$, we will choose a closed interval $I_v=I^i_v$ of $S^1$ mapped to $v$ by $\sigma_i$. (These intervals will be used in subsequent steps to accommodate the parts of the graph not yet in the image of $\sigma_i$). Then, at step $i+1$, we will pick a suitable set of finite circuits in $E(G) - F_i$, put them into $F_i$ to obtain $F_{i+1}$, and modify $\sigma_i$ to $\sigma_{i+1}: S^1 \to \overline{F_{i+1}}$. We might also add some vertices to $S_i$ to obtain $S_{i+1}$. 

To begin with, let $F_0=E(C)$, $S_0=\emptyset$ and $\sigma_0$ as defined above. Let $e_1,e_2,\ldots$ be an enumeration of the edges of $G$. 
Then, for $i=1,2,\ldots$, perform a step of the following type. In step $i$, let for a moment $S_i = S_{i-1}$ and consider the components of $G - F_{i-1}$, which are finitely many since $F_{i-1}$ is finite. For every such component, say $D$, there is by the inductive hypothesis at most one vertex $v=v_D \in S_i$ incident with $D$. If there is none, then just pick any vertex $v$ incident with both $D$ and $F_{i-1}$, put it into $S_{i}$, and let $I_v$ be any maximal interval of $S^1$ mapped to $v$ by $\sigma_{i-1}$; recall that by the construction of the $\sigma_{i}$, $I_v$ is non-trivial. 

We claim that every edge in $G - F_{i-1}$ lies in a finite cycle. Indeed, if some edge $e=wz \in E(G) \sm F_{i-1}$ does not, then $\{e\} \cup F_{i-1}$ separates $w$ from $z$ in $G$. But then, let $\{W,Z\}$ be a bipartition of the vertices of $G$ such that $w\in W$, $z\in Z$, and there are no edges between $W$ and $Z$ in $G - (\{e\} \cup F_{i-1})$. Let $E(W,Z)$ be the set of edges in $G$ between $W$ and $Z$. Since $F_{i-1}$ is a finite edge-disjoint union of finite circuits, $|E(W,Z) \cap F_{i-1}|$ is even; thus $E(W,Z) \ni e$ is an odd cut of $G$, which contradicts our assumption. 

Similarly, if $D$ is any component of $G - F_{i-1}$ and we let  $j=j_D := \min \{k \mid e_k \in E(D)\}$, then $D$ contains a finite edge-set $C_D$ in $D$ such that $C_D$ is incident with $v_D$, it contains the edge $e_j$, and $C_D=C_1 \cup \ldots \cup C_k$ where the $C_i$ are pairwise disjoint finite circuits and $C_i$ is incident with $C_{i-1}$ \fe\ $1< i\leq k$. Indeed, if $e_j$ happens to be incident with $v_D$ then we let $C_D$ consist of a single circuit containing $e_j$, and if not we pick a shortest \pth{v_D}{e_j}\ $P$, let $C_1$ be a finite circuit in $D$ containing the first edge of $P$, let $C_2$ be a finite circuit  in $D - C_1$ containing the first edge of $P$ not in $C_1$,  and so on; such a $C_2$ exists by a similar argument as in the previous paragraph. If $v$ has infinite degree then we slightly modify the choice of $C_D$ so that, in addition to the above requirements, $C_D$ contains at least $4$ edges incident with $v$. Let $F_i$ be the union of $F_{i-1}$ with all these edge-sets $C_D$, one for each component $D$ of $G - F_{i-1}$.

Then, to obtain $\sigma_i$ from $\sigma_{i-1}$, for each component $D$ of $G - F_{i-1}$, let $\sigma_i$ map $I_v=I^{i-1}_{v_D}$ continuously to $C_D$, in such a way that every edge in $C_D$ is traversed precisely once, and each time a point $x\in I_v$ is mapped to some vertex $w$ (incident with $C_D$) there is a (closed) non-trivial subinterval $J \ni x$ of $I_v$ such that every point of $J$ is mapped to $w$; however, make sure that each such subinterval $J$ has length at most $2^{-i}$. This defines $\sigma_i$. To complete step $i$, we still have to define the interval $I^i_v$ \fe\ $v\in S_i$. If step $i$ did not affect $v$ yet, that is, if we did not modify the image of $I^{i-1}_v$ when defining $\sigma_i$ from $\sigma_{i-1}$ (which happens if all edges of $v$ were in $F_{i-1}$), then we let $I^i_v:= I^{i-1}_v$. If we did  modify the image of $I^{i-1}_v$, then we let $I^i_v$ be one of the maximal (closed) subintervals of $I^{i-1}_v$ mapped by $\sigma_i$ to $v$; such an interval always exists and is non-trivial by the construction of $\sigma_i$. If, in addition, $v$ has infinite degree, then we have to be a bit more careful with our choice of an interval $I^{i}_{v}$: by the choice of the $C_D$ there are at least $4$ edges in some $C_D$ incident with $v=v_D$, and so there is an {\em inner} maximal subinterval $I'$ of $I_v$ mapped to $v$ by $\sigma_i$; we let $I^{i}_{v}:=I'$. This completes step $i$.

We claim that for every point $x \in S^1$ the images $\sigma_i(x)$ converge to a point in $\Getop$. Indeed, since $\Getop$ is compact, $(\sigma_i(x))_{i\in \N}$ has a subsequences converging to a point $y$ in \Getop. Moreover, by the choice of the $C_D$, for every edge $e_j\in E(G)$ there is an $i$ so that $e_j\in F_m$ holds for all $m\geq i$. Thus, for every basic open set $O$ of $y$ there is an $i$ such that the component $C$ of $G - F_i$ in which $y$ lives is a basic open set of $y$ (up to some half edges) contained in $O$, and so $C$ contains an element $p$ of $(\sigma_i(x))_{i\in \N}$. By the definition of the $\sigma_{i}$, if $x$ is mapped to a point $p$ by $\sigma_{i+1}$, then for all steps succeeding step $i+1$ the image of $x$ will lie in the component of $G - F_i$ that contains $p$. This means that $C$, and thus $O$, contains all but finitely many members of $(\sigma_i(x))_{i\in \N}$, and so the whole sequence $(\sigma_i(x))_{i\in \N}$ converges to $y$.

Hence we may define a map $\sigma: S^1 \to \Getop$ mapping every point $x \in S^1$ to a limit of $(\sigma_i(x))_{i\in \N}$. 

Our next aim is to prove that $\sigma$ is continuous. For this, we have to show that the preimage of any basic open set of $\Getop$ is open. This is obvious for basic open sets of inner points of edges. For every other point $y\in \EOO\cup V$, the sequence of basic open sets of $y$ that arise after deleting  $F_i, i \in \N$ is converging, so it suffices to consider the basic open sets of that form, and it is straightforward to check that their preimages are indeed open. Thus  $\sigma$ is continuous, and since by construction it traverses each edge exactly once it is a \tet.

Call two points in $\EOO\cup V$ \defi{equivalent} if they cannot be separated by a finite edgeset. Call a point $x\in S^1$ a \defi{hopping point} if $x$ lies in an interval of the form $I^i_v$ for every step $i$ (but perhaps for a different vertex $v$ in different steps). We now claim that 
\begin{equation} \label{star} \begin{minipage}[c]{0.85\textwidth}
for every two distinct hopping points $x,y \in S^1$, $\sigma(x)$ and $\sigma(y)$ are not equivalent.
\end{minipage}\ignorespacesafterend \tag{\ensuremath{*}} \end{equation}
Indeed, for any point $p\in \Getop$, there is in every step $i$ at most one vertex $v$ in $S_i$ meeting the component $C$ of $G -F_i$ in which $p$ lives. Moreover, all points equivalent to  $p$ also live in $C$. Thus $I_v$ is the only interval of $S^1$ in which $p$ and its equivalent points can be accommodated. Since $I_v$ gets subdivided after every step, there is only one point of $S^1$ that can be mapped by $\sigma$ to a point equivalent to $p$.

Since by construction the only points of $S^1$ that can be accumulation points of preimages under \sig\ of vertices are the hopping points, the second part of the assertion of \Lr{etoconv} follows from~\eqref{star}.
\end{proof}

\begin{proof}[Proof of \Tr{Lg}]
Let \g be a countable graph such that \eti\ has a \tet. Then \g has no finite cut of odd cardinality by \Lr{etoc}; thus \Getop\ has a \tet\ \sig\ as provided by \Lr{etoconv}.

Let $v$ be a vertex of infinite degree. It follows easily by the choice of the circuits $C_D$ and the intervals $I^{i}_{v_D}$  in the proof of \Lr{etoconv}, that whenever \sig\ runs into $v$ along an edge then it must also leave $v$ along an edge (rather than along an arc containing a double ray); formally, this fact can be stated as follows:
\begin{equation} \label{club} \begin{minipage}[c]{0.85\textwidth}
For any interval $I$ of $S^1$ with $\sigma(I)=\{v\}$, if there is an interval $I'\supset I$ of $S^1$ such that $\sigma(I')\subseteq e$ where $e$ is an edge (incident with $v$) in \eti, then there is an interval $I''\supset I'$ of $S^1$ such that $I'' \sm I$ is disconnected and $\sigma(I'')\subseteq e\cup e'$, where $e'$  is also an edge in \eti.
\end{minipage}\ignorespacesafterend \tag{\ensuremath{\diamondsuit}} \end{equation}
(We stated~\eqref{club} only for vertices of infinite degree because it is only interesting for such vertices; it is trivially true for vertices of finite degree).

We are now going to transform \sig\ into a \hcy\ $\tau$ of $\mtopx{L(G)}$.
Note that if a set $F\subseteq E(G)$ converges in \eti\ then $F\subseteq V(L(G))$ also converges in $\mtopx{L(G)}$; to see this, recall that the basic open sets in \eti\ are defined by removing finite edge-sets, while the \bos s in $\mtopx{L(G)}$ are defined by removing finite vertex sets. Thus we can define a function $\pi$ mapping any edge-end \oo\ of $G$ to the end of $L(G)$ to which the edge-set of any ray in \oo\ converges, and it is straightforward to check that $\pi$ is well defined. Moreover, for any vertex $v$ of infinite degree in \g let $\pi'(v)$ be the accumulation point of $E(v)$ in $\mtopx{L(G)}$.


We now transform \sig\ into a new mapping $\sig'$, which we will then slightly modify to obtain the required \hcy\ in $\mtopx{L(G)}$. Let $\sigma': S^1 \to |L(G)|$ be the mapping defined as follows:
\begin{itemize}
\item $\sigma'$ maps the  preimage under $\sigma$ of each edge $e\in E(G)$ to $e\in V(L(G))$;
\item for each interval $I$ of $S^1$ mapped  by $\sigma$ to a trail $u e y e' w$ where $u,y,w$ are vertices, define $I'$ to be the (non-trivial) subinterval of $I$ mapped by $\sigma$ to $y$, and let $\sigma'$ map $I'$ continuously and bijectively to the edge $ee' \in E(L(G))$;
\item if $\sigma(x)=\oo \in \EOO(G)$ then let $\sigma'(x)=\pi(\oo)$.
\item if $\sigma(x)=v$ where $v$ is a vertex of infinite degree in \g and $x$ does not lie in an interval $I$ of $S^1$ mapped by $\sigma$ to a trail $u e v e' w$ (which can be the case if $x$ is a hopping point), then let $\sigma'(x)=\pi'(v)$.
\end{itemize}

It follows by the construction of $\sigma'$ and~\eqref{club} that the image of $\sig'$ does not contain any vertex of \G, in other words, that $\sig'(S^1)\subseteq \mtopx{L(G)}$.

By construction, $\sigma'$ maps a non-trivial interval to every vertex it traverses, which we do not want since a \hcy\ must be injective; however, it is easy to modify $\sigma'$ locally to obtain a mapping $\tau: S^1 \to |L(G)|$ that maps a single point to each vertex.

It follows easily by the continuity of \sig\ and the definition of $\pi$ and $\pi'$ that $\sigma'$, and thus $\tau$, is continuous. Since \sig\ is a \tet\ of $G$, $\tau$ visits every vertex of $L(G)$ precisely once. Moreover,  the second part of the assertion of \Lr{etoconv} implies that $\tau$ visits no end more than once, which means that it injective. Since $S^1$ is compact and $|L(G)|$ Hausdorff, \Lr{injHom} implies that $\tau$ is a homeomorphism, and thus a \hcy\ of $\mtopx{L(G)}$.
\end{proof}

Having seen \Tr{Lg} and the discussion in \Sr{introLg}, it is natural to ask if the following is true: for every graph $G$ and \lER, if \ltp\ has a \tet\ then \ltopxl{\lL}{L(G)}\ ---as defined in \Sr{introLg}--- has a \hcy.
This is however not the case: suppose $G$ and $\ell$ are such that \ltp\ has a \tet\ and the boundary of \ltopxl{\lL}{L(G)}\ contains a subspace $I$ homeomorphic to the unit interval ---which can easily happen, see \Tr{Lind}. Then, \ltopxl{\lL}{L(G)}\ cannot have a \hcy\ $\tau$, because the preimage of $I$ under $\tau$ must be totally disconnected and $\tau$ would then induce a homeomorphism between $I$ and a totally disconnected set.

\note{
\section{Euler tours} \label{secet}

\begin{conjecture}
\ltp\ has a \tet\ \iff\ \ltp\ is compact and \fcg\ has a \tet.
\end{conjecture}

This is true for the hyperbolic compactification of a hyp. planar Cayley graph: any finite plane eulerian graph has a decomposition of its edges into face boundaries, and by compactness this is also true for infinite graphs. Take such a decomposition, pick a cycle, and insert cycles recursively at the right vertices to get a \tet\ in the hyperbolic compactification.
}

\note{
	\section{Stochastics} \label{secSto}
\subsection{Percolation-like}

You could try to assign lengths to the edges of, say, a binary tree randomly according to some rule with some parameter and look for phase transition 		
	phenomena as far as the size of the boundary of \ltop\ is concerned.
}

\comment{
	\section{Future}
	\subsection{Brownian motion}

When studying random walks on graphs the basic question is whether the random walk is transient or recurrent, i.e.\ whether it escapes to infinity or not \cite{woessBook}. With \ltop\ we might be able to have random walks going to infinity and coming back: a random walk can be modelled by Brownian motion on the graph, where edges have been given appropriate lengths so that the transition probabilities between vertices are the desired ones \cite{LyonsBook}. Now if these edge lengths \el\ give rise to topological paths of finite length in \ltop\ between a vertex and a \bp,  the \bp\ might be hit after a finite amount of time. Could we define a Brownian motion on \ltp\ that can visit a \bp\ and come out of it using some topological path? This seems to me to be hard to do in general, but could be doable if \g is, say, a Cayley graph.

}

\section{Outlook}

In this paper we studied several aspects of \ltop, proving basic facts and  discussing applications. I expect that the current work will lead to interesting research in the future, and I hope that other researchers will contribute.

Some open problems were suggested here, but there are also other directions in which further work could be done. The general theory developed in this paper may offer a framework for other specific compactifications of infinite graphs that, next to the Freudenthal and the hyperbolic compactification ---see \Sr{subSpec}--- have applications in the study of infinite graphs and groups. A further possibility would be to study Brownian motion on the space \ltp, and seek interactions with the study of electrical networks as discussed in \Sr{inele}. Finally, the new homology described in \Sr{sechom} suggests interactions between the study of infinite graphs and other spaces, e.g.\ like  in \Cnr{conjMcL}, that may be worth following up.

\acknowledgement{I am grateful to the research group of Reinhard Diestel, especially to Henning Bruhn, Johannes Carmesin, Matthias Hamann, Fabian Hundertmark, Julian Pott, Philipp Spr\"ussel and Georg Zetzsche, for proof--reading this paper and making valuable suggestions. I am also grateful to the student in Haifa who pointed out to me the proof of \Cr{corcomp}. I regret I did not ask his name.}
 
\bibliographystyle{plain}
\bibliography{collective}

\end{document}